\documentclass{article}
\usepackage{amsmath, amssymb, amsfonts, amsthm, xcolor,biblatex,authblk}
\usepackage{geometry}
\usepackage[utf8]{inputenc}
\usepackage{graphicx}
\usepackage[hidelinks]{hyperref}
\usepackage{comment}

\newtheorem{theorem}{Theorem}[section]

\newtheorem{definition}[theorem]{Definition}

\newtheorem{corollary}[theorem]{Corollary}
\newtheorem{assumption}[theorem]{Assumption}
\newtheorem{remark}[theorem]{Remark}

\newcommand{\R}{\mathbb{R}}

\newcommand{\Z}{\mathbb{Z}}
\newcommand{\eps}{\varepsilon}
\newcommand{\Mzero}{M_0}
\newcommand{\Meps}{M_\eps}

\newcommand{\PiS}{\Pi_S}
\newcommand{\PiR}{\Pi_R}
\newcommand{\PiN}{\Pi_N}
\newcommand{\PiNs}{\Pi_{N_s}}
\newcommand{\PiNu}{\Pi_{N_u}}
\newcommand{\PiM}{\Pi_M}
\newcommand{\Sone}{\mathbb{S}^1}
\newcommand{\pderiv}[2]{\frac{\partial #1}{\partial #2}}

\newcommand{\sx}{s_x}
\newcommand{\st}{s_t}
\newcommand{\cx}{c_x}
\newcommand{\ct}{c_t}
\newcommand{\tx}{t_x}
\newcommand{\cp}{c_{\phi}}
\newcommand{\sphi}{s_{\phi}}

\title{
\textbf{A coordinate-independent  Pontryagin-Rodygin theorem for slow-fast averaging}}
\author[1]{Bob Rink}
\author[2]{Theodore Vo}
\author[3]{Martin Wechselberger\thanks{Corresponding author: martin.wechselberger@sydney.edu.au}}

\affil[1]{\footnotesize Department of Mathematics, Vrije Universiteit, 1081 HV Amsterdam, the Netherlands}
\affil[2]{\footnotesize School of Mathematics, Monash University, Clayton, Victoria 3800, Australia}
\affil[3]{\footnotesize School of Mathematics and Statistics, University of Sydney, Sydney NSW, Australia}

\date{}

\begin{document}

\maketitle

\begin{abstract}
The slow drift along a manifold of periodic orbits is a key mathematical structure underlying bursting dynamics in many scientific applications. While classical averaging theory,  as formalised by the {\em Pontryagin-Rodygin theorem}, provides a leading-order approximation for this slow drift, the connection to the underlying geometry described by {\em Geometric Singular Perturbation Theory} (GSPT)--also known as {\em Fenichel theory}--is often not explicit, particularly at higher orders. This paper makes that connection rigorous and constructive using the parametrisation method. We provide a detailed, self-contained exposition of this functional analytic technique, showing how it synthesizes the geometric insight of invariant manifold theory with a systematic, perturbative algorithm. By treating the manifold's embedding and the reduced flow as coupled unknowns, the method provides a constructive proof of an averaged system that is guaranteed to be geometrically consistent with the persistence of the normally hyperbolic manifold to any order. We translate the abstract theory into a concrete computational procedure using Floquet theory, spectral analysis, and the Fredholm alternative, yielding a practical guide for computing high-accuracy, higher-order averaged models, and we demonstrate its implementation, both analytically and numerically, through specific examples.
\end{abstract}

\section{Introduction}

A central challenge in the study of complex systems is to understand how macroscopic, emergent dynamics arise from intricate underlying rules. In systems with multiple timescales, this behavior is governed by the flow on a lower-dimensional {\em slow manifold} embedded in the full state space. The identification of this manifold and the derivation of the reduced dynamics on it is a fundamental goal of model reduction, with a rich history of approaches.

One school of thought is rooted in {\em centre manifold theory} \cite{Carr1981}, which provides a rigorous framework for analysing bifurcations near, e.g., equilibria by systematically reducing the dynamics to the subspace spanned by eigenvectors with non-hyperbolic eigenvalues. A related but distinct concept is {\em invariant manifold theory} \cite{HPS1977} for singularly perturbed systems, as laid out in, e.g., Geometric Singular Perturbation Theory (GSPT) \cite{fenichel1979, jones1995}. Here, the manifold is not restricted to a single equilibrium but persists over a region of phase space, representing a family of equilibria for the limiting fast dynamics known as the {\em layer problem}. This framework is essential for studying general, {\em `non-standard'} singular perturbation problems \cite{wechselberger2020}. Parallel to these geometric theories, {\em algorithmic methods} such as Computational Singular Perturbation (CSP) were developed to numerically identify the slow and fast subspaces from simulation data, providing a powerful diagnostic tool, particularly in chemical kinetics \cite{lam1989}.

This paper focuses on the {\em parametrisation method}, a technique that synthesises the geometric insight of invariant manifold theory with a constructive, functional analytic framework \cite{haro2016, roberts2015}. Like the methods mentioned above, it provides a coordinate-independent description of the slow manifold.  The idea of parametrization has been applied before in  GSPT problems involving manifolds of equilibria \cite{lizarraga2021}, and it has several unique features that make it particularly powerful for the periodic orbit case considered here:
\begin{itemize}
    \item It is explicitly a {\em perturbative method}, constructing the manifold and its flow as a formal power series in a small parameter $\eps$. This provides a systematic path to achieving high-order accuracy.
    \item It is an {\em equation-based, functional approach}, directly solving for a function that defines the manifold's embedding. This contrasts with data-driven methods like CSP, which iteratively identify local tangent spaces.
    \item It rigorously treats the manifold's embedding and its reduced flow as {\em coupled unknowns}, solving for them simultaneously to ensure geometric consistency at all orders.
\end{itemize}

The main contribution of this manuscript is to apply this method to a key problem in GSPT: the slow drift along a manifold foliated by periodic orbits. This mathematical structure is the cornerstone of {\em bursting dynamics}, a ubiquitous phenomenon observed in many applications, in particular in neuroscience \cite{Rinzel1987,Izhikevich2000,Rubin2002}. The analysis of such systems has a rich history, rooted in classical averaging techniques \cite{Bogoliubov1961, sanders2007}.  The foundational result in this area is the Pontryagin-Rodygin theorem \cite{PontryaginRodygin1960}, which provides conditions for the existence of a first-order averaged system that approximates the slow drift. However, the classical theorem is presented in a coordinate-dependent fashion, and the connection between the formal averaging calculation and the underlying geometry described by GSPT is often not explicit, particularly when seeking higher-order accuracy.

 This manuscript makes this connection rigorous and, in doing so, provides a coordinate-independent and higher-order generalisation of the Pontryagin-Rodygin theorem. We show how the parametrisation method provides a constructive proof of an averaged system that is guaranteed to be geometrically consistent with the persistence of a {\em normally hyperbolic} manifold to any order in $\eps$. We thus fill a theoretical gap and establish a practical, systematic procedure for computing higher-order averaged models. We translate the abstract theory into a concrete algorithm using Floquet theory, spectral analysis, and the Fredholm alternative to separate the slow dynamics from the manifold's geometry.
 Establishing this rigorous, high-order framework for the normally hyperbolic case is a crucial first step towards a rigorous analytical framework for systems where normal hyperbolicity is lost,  which can lead to complex phenomena such as torus canards \cite{vo2017,roberts2015averaging}.



The paper is structured as follows. In Section \ref{sec:problem_setup}, we define the problem and its geometric properties. In Section \ref{sec:para}, we detail the iterative procedure of the parametrisation method. Section \ref{sec:results} presents our main theoretical results, formalising the existence of the averaged flow and relating our approach to classical averaging theory. We then demonstrate the theory by way of two examples. In Section~\ref{sec:ellipsoid}, we detail the method in an analytically tractable non-standard multiple timescale system.  Finally, in Section~\ref{sec:mlt}, we describe the numerical implementation of the method in a standard slow/fast bursting model from neuroscience. 

\section{The problem setup}\label{sec:problem_setup}
We study perturbation problems of the form
\begin{equation} \label{eq:intro_system}
    \frac{dz}{dt} = z' = F(z, \eps)=\sum_{i=0}^m \eps^i F_i(z), \quad z \in \R^n, \quad 0 < \eps \ll 1, \quad m\ge 1
\end{equation}
with $F(z,\eps)$ to be sufficiently smooth in both arguments.
We consider this system \eqref{eq:intro_system} under the following assumption.

\begin{assumption}\label{ass:unperturbed}
The unperturbed system,  
\begin{equation}\label{eq:layer_problem}
    z' = F(z, 0) = F_0(z)\,,
\end{equation}
called the layer problem, possesses a $k$-dimensional family of periodic orbits $\gamma_0(x,t)$, $1\le k<n-1$. This family forms an embedded $(k+1)$-dimensional manifold, $\Mzero$.
\end{assumption}
\noindent
We shall denote by $\tau(x)>0$ the period of the periodic orbit $t\mapsto \gamma_0(x,t)$, and by $\omega(x):=\frac{2\pi}{\tau(x)}$ its frequency. An embedding of the manifold $M_0$ is then given by
$$\Gamma_0: U \times \mathbb{S}^1 \to \R^n \ \mbox{defined by}\ 
\Gamma_0(x,\phi) := \gamma_0(x,\omega(x)^{-1}\phi) \, .
$$
Here $U \subset \R^k$ is an open set parametrising the periodic orbits in $M_0$, while $\mathbb{S}^1 = \R / (2\pi\Z)$ is the unit circle parametrising the phase $\phi = \omega(x) t$ along the orbits. 
\begin{assumption}
    The period map $x\mapsto \tau(x)$ is smooth and uniformly bounded from below and above. The same is then true for the frequency map $x\mapsto \omega(x)$.
\end{assumption}
By definition, the map $D\Gamma_0(x,\phi): \mathbb{R}^k\times \mathbb{R} \to \mathbb{R}^n$ has full rank $(k+1)$.
Moreover, in the parametrisation chart $U\times \mathbb{S}^1$, the vector field $F_0$ on $M_0$ is given by $r_0(x,\phi)=(0,\omega(x))^\top$, i.e.,
\begin{equation}
    \begin{aligned}
    x'&=0\,,\\
    \phi' &= \omega(x)\, .
\end{aligned}
\end{equation}
In other words, our embedding  $\Gamma_0$ satisfies the {\em conjugacy equation}
\begin{equation}\label{eq:conjugacy}
    D\Gamma_0 (x,\phi) r_0 (x,\phi) = 
\partial_\phi \Gamma_0 (x,\phi) \omega(x) =
F_0(\Gamma_0(x,\phi))\,.
\end{equation}

\noindent
A sketch of an embedding of a manifold of periodic orbits $M_0$ is shown in Figure~\ref{fig:embed-averaging}.
\begin{figure}[t]
    \centering
    \includegraphics[width=11.0cm]{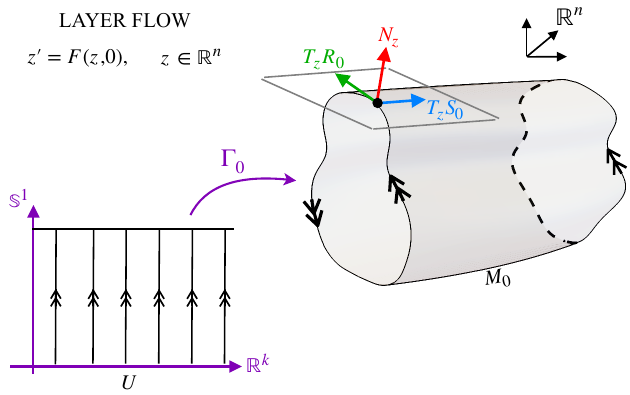}
    \caption{ A sketch of a normally hyperbolic manifold $M_{0}$ foliated by periodic orbits embedded by $\Gamma_0$. The tangent space $T_zM_0$ at a point $z\in M_0$ splits into the direction of the flow, $T_zR_0$, and the phase-preserving directions, $T_z S_0$. The dynamics normal to the manifold occur in the bundle $N_z$. 
    }
    \label{fig:embed-averaging}
\end{figure}

\subsection{Variational equation}

Given the limiting system \eqref{eq:layer_problem} and its family of periodic orbits $\gamma_0(x,t)$, we are interested in the behavior of perturbations around these orbits, i.e., we consider a perturbed solution $z(t) = \gamma_0(x,t) + \zeta(x,t)$, where $\zeta(x,t)$ is considered a small deviation, i.e. $||\zeta||\ll 1$. 
This leads to the leading order {\em variational equation}:
\begin{equation}\label{eq:var_eq}
    \frac{\partial}{\partial t}\zeta =D_t\zeta = DF_0(\gamma_0(x,t))\zeta=:A(x,t)\zeta\,.
\end{equation}
This is a linear, non-autonomous ordinary differential equation, where $A(x,t)$ is a $\tau(x)$-periodic matrix, i.e., $A(x,t)=A(x,t+\tau(x))$. 

\begin{definition}
    For fixed $x$, let $\Phi(x,t)$ be the fundamental matrix solution to this variational equation \eqref{eq:var_eq}, and let 
    $$\Phi(x,t,s):=\Phi(x,t)\Phi(x,s)^{-1}$$ 
    be its corresponding transition matrix, i.e., $\Phi(x,t_0,t_0)= \mathbb{I}_n$ and
    $$\frac{\partial}{\partial t}\Phi(x,t,t_0)=D_t \Phi(x,t,t_0)=A(x,t)\Phi(x,t,t_0)\,.$$
    Hence for given initial condition $\zeta(x,t_0)=\zeta_0(x)$, $\zeta(x,t)=\Phi(x,t,t_0)\zeta_0(x)$ is a solution of \eqref{eq:var_eq}.
\end{definition} 

\begin{remark}
    With a slight abuse of notation, we identify from now on the local coordinate $\zeta$ with $z$.
\end{remark}
\subsection{Dichotomy-type definition of normal hyperbolicity}

\begin{assumption}\label{ass:hyperbolicity}
The invariant manifold $\Mzero$ is normally hyperbolic. 
\end{assumption}
Normal hyperbolicity is a fundamental dynamical concept in the analysis of invariant manifolds. It captures dominant transverse motion towards or away from the invariant manifold relative to the tangential motion along the manifold, i.e., the dynamics of the variational equation \eqref{eq:var_eq}  exhibit an {\em exponential dichotomy} in directions \textit{transverse} to $M_0$, while remaining {\em `neutral'} (i.e., non-exponential) in directions \textit{tangent} to $M_0$.
Specifically, for each periodic orbit $\gamma_0(x,t)$ within $M_0$, there exist two $\Phi(x,t,t_0)$-invariant subspaces of $\mathbb{R}^n$, denoted $T_z M_0$ and ${\mathcal N}_z$, $z=\Gamma_0(x,\phi) \in M_0$:
\begin{enumerate}
    \item \textbf{Normal Space ${\mathcal N}_z = {\mathcal N}^s_z \oplus {\mathcal N}^u_z$:} This $(n-(k+1))$-dimensional invariant (dynamic) subspace, obliquely transverse (i.e., transverse but not necessarily orthogonal) to $M_0$, splits into stable (${\mathcal N}^s_z$) and unstable (${\mathcal N}^u_z$) invariant sub-bundles.
    \begin{itemize}
        \item For vectors $v^s \in {\mathcal N}^s_z$, solutions exponentially contract towards $M_0$. There exist constants $K \ge 1$ and $\lambda_s > 0$ such that:
            $$||\Phi(x,t,t_0) v^s|| \le K e^{-\lambda_s(t-t_0)} ||v^s||\,,\quad \mbox{for }t\ge t_0\, .$$
        \item For vectors $v^u \in {\cal N}^u_z$, solutions exponentially expand away from $M_0$. There exist constants $K \ge 1$ and $\lambda_u > 0$ such that:
            $$||\Phi(x,t,t_0) v^u|| \le K e^{\lambda_u(t-t_0)} ||v^u||\,,\quad \mbox{for }t\le t_0\,.$$
    \end{itemize}
    The constant $\lambda=\min \{\lambda_s,\lambda_u\}$ is the `normal hyperbolicity' exponent, and it is bounded away from zero.
    \item \textbf{Tangent Space $T_z M_0$:} This $(k+1)$-dimensional invariant subspace corresponds to the `neutral' dynamics along and within the manifold $M_0$. Solutions starting in $T_z M_0$ remain in $T_{\Phi(x,t,t_0)z} M_0$ and do not exhibit exponential growth or decay relative to each other. 
\end{enumerate}
Normal hyperbolicity guarantees the persistence of the invariant manifold and its foliation under small perturbations: 
\noindent
\begin{theorem}[Fenichel \cite{fenichel1979}: Persistence of the Invariant Manifold] \label{thm:Fenichel}
For a sufficiently small $0< \varepsilon\ll 1$, there exists an invariant manifold $M_\varepsilon$ for the full system \eqref{eq:intro_system} that is $C^r$-diffeomorphic to $M_0$ and lies an $O(\varepsilon)$ distance from $M_0$. The dynamics on $M_\varepsilon$ are an $O(\varepsilon)$ perturbation of the dynamics on $M_0$. 
\end{theorem}

\begin{remark}
In our setup, this persistent invariant manifold $M_\varepsilon$ is also foliated by `periodic' fibres which form an invariant foliation of the flow on $M_\eps$. Note, the family of `periodic' fibres (closed fibres) is invariant under the flow, i.e., fibres are mapped to fibres under the flow, but the individual fibres are not invariant--a consequence of Fenichel theory. This slow drift across the fibres is precisely the averaged flow that this paper aims to construct.
\end{remark}

\subsection{Monodromy matrix and the normal bundle ${\cal N}=\cup_{z\in M_0} {\cal N}_z$}\label{sec:mono-and-normal-bundle}

\begin{definition}
    For fixed $x\in\mathbb{R}^k$, the monodromy matrix is defined as 
$$
M(x,t_0) := \Phi(x,t_0+\tau(x),t_0)\,,
$$ i.e., the evolution of the fundamental matrix solution over one period.  
\end{definition}
\noindent
Floquet's theorem states that 
\begin{equation}\label{eq:Floquet}
\Phi(x,t,t_0) = P(x,t)e^{B(x)(t-t_0)}P^{-1}(x,t_0)\,,   
\end{equation}
where $P(x,t)$ is a regular $\tau(x)$-periodic square matrix, and $B(x)$ is a constant (in time) matrix.
\begin{definition}\label{def:FloquetExp}
    The eigenvalues of $B(x)$ are the Floquet exponents, $\mu_j(x),\,j=1,\ldots n$. The corresponding eigenvalues of the monodromy matrix $$
    M(x,t_0)=P(x,t_0)e^{B(x)\tau(x)}P^{-1}(x,t_0)
    $$ are the Floquet multipliers, $\rho_j(x)=e^{\mu_j(x)\tau(x)},\,j=1,\ldots n$.
\end{definition}
\begin{remark}
    From \eqref{eq:Floquet} it is clear that the eigenvalues (Floquet exponents) of $B(x)$ are independent of the initial time $t_0$, and we also see from Definition~\ref{def:FloquetExp} that the family of monodromy matrices $M(x,t_0)$ is similar. Hence, the choice of initial time $t_0$ is irrelevant, and we have
    $$
    \Phi(x,t+\tau(x),t_0)=\Phi(x,t,t_0)M(x,t_0)\,.
    $$
\end{remark}
Assumption~\ref{ass:hyperbolicity} directly translates to the properties of these Floquet exponents/multipliers:
\begin{itemize}
    \item There are exactly $(k+1)$ {\em trivial} Floquet exponents $\mu_j(x)=0$ which correspond to the $(k+1)$ Floquet multipliers $\rho_j(x)=1$, $j=1,\ldots k+1$. The generalized eigenspace associated with these trivial exponents is precisely the tangent space $T_z M_0$ described above.
    \item The remaining $n-(k+1)$ Floquet exponents have non-zero real parts $\text{Re}(\mu_j(x)) \ne 0$, $j=1,\ldots n-(k+1)$.
        Those with $\text{Re}(\mu_j(x)) < 0$ correspond to the stable normal bundle ${\cal N}^s$. Their corresponding Floquet multipliers have $|\rho_j(x)| < 1$.
        Those with $\text{Re}(\mu_j(x)) > 0$ correspond to the unstable normal bundle ${\cal N}^u$. Their corresponding Floquet multipliers have $|\rho_j(x)| > 1$.
\end{itemize}
The constant $\lambda$ in the exponential dichotomy is $\min_{j} |\text{Re}(\mu_j(x))|$ for the non-trivial Floquet exponents.
An important relationship is obtained by differentiating \eqref{eq:Floquet} with respect to time $t$ which yields a differential equation for the periodic matrix $P(x,t)$:
\begin{equation}\label{eq:Floquet-periodic}
D_t P(x,t) + P(x,t)B(x) = DF_0(\gamma_0(x,t)) P(x,t)\,.
\end{equation}

\subsubsection{Basis of $T_zM_0$ (tangential bundle) and ${\cal N}_{z}$ (normal bundle)}
Let $b_{M_i}(x),\,i=1,\ldots (k+1)$ denote a basis of generalised eigenvectors for the trivial exponents of $B(x)$, and let 
$$
    b_M(x)=(b_{M_1}(x),\ldots b_{M_{k+1}}(x))
    $$ 
be the corresponding $n\times (k+1)$ matrix with the $b_{M_i}$'s as column vectors, i.e., we have 
$$
B(x)b_M(x)=b_M(x)J_M(x)
$$ where $J_M(x)$ is a $(k+1)$ square matrix, the corresponding Jordan-block matrix.

\noindent
Similarly, let
$b_{N_j}(x),\,j=1,\ldots n-(k+1)$ denote a basis of generalised eigenvectors for the non-trivial exponents of $B(x)$, and let 
$$
b_N(x)=(b_{N_1}(x),\ldots b_{N_{n-(k+1)}}(x))
$$ 
be the corresponding $n\times (n-(k+1))$ matrix with the $b_{N_j}$'s as column vectors, i.e., we have 
$$B(x)b_N(x)=b_N(x)J_N(x)$$ 
where $J_N(x)$ is a $n-(k+1)$ square matrix, the corresponding Jordan-block matrix.
%
Since both spaces are invariant we have
$$
    B(x)b(x)=b(x)J(x)
    \quad\mbox{with}\quad
    b(x)= (b_M(x)\,b_N(x))
    \quad\mbox{and}\quad
    J(x)=
    \begin{pmatrix}
     J_M(x) & O \\
     0 & J_N(x)
    \end{pmatrix}
    \,.
$$
%
%
We use the basis matrix $b(x)$ 
to rewrite \eqref{eq:Floquet} as
$$
\Phi(x,t,t_0) = P(x,t) b(x)e^{J(x)(t-t_0)}b^{-1}(x)P^{-1}(x,t_0)
$$
and define composite $\tau(x)$-periodic sub-matrices in $t$
$$
\begin{aligned}
    v_M(x,t) &:=P(x,t)b_M(x)=\Phi(x,t,t_0) v_M(x,t_0) e^{-J_M(x)(t-t_0)}\,,\quad 
v_M(x,t_0)=P(x,t_0)b_M(x)\,,\\
v_N(x,t) &:=P(x,t)b_N(x)=\Phi(x,t,t_0) v_N(x,t_0) e^{-J_N(x)(t-t_0)}\,,\quad 
v_N(x,t_0)=P(x,t_0)b_N(x)\,.
\end{aligned}
$$
These describe the transport of the basis vectors $v_{M_i}(x,t)$ between tangent spaces $T_{z}M_0$ along the periodic orbit $\gamma_0(x,t)$, respectively the transport of the basis vectors $v_{N_j}(x,t)$ between normal spaces ${\cal N}_{z}$ along the periodic orbit $\gamma_0(x,t)$. 
%
%
We can also define them in $(x,\phi)$-coordinates:
$$
\begin{aligned}
    V_M(x,\phi) &:= v_M(x,\omega^{-1}\phi)=
P(x,\omega^{-1}\phi)b_M(x)\,,\quad 
V_M(x,\phi_0)=v_M(x,\omega^{-1}\phi_0)\,,\\
V_N(x,\phi) &:= v_N(x,\omega^{-1}\phi)=
P(x,\omega^{-1}\phi)b_N(x)\,,\quad 
V_N(x,\phi_0)=v_N(x,\omega^{-1}\phi_0)\,.
\end{aligned}
$$
The tangential and normal `field' then fulfills \eqref{eq:Floquet-periodic},
$$
\begin{aligned}
    D_t v_M(x,t) + v_M(x,t)J_M(x) = DF_0(\gamma_0(x,t)) v_M(x,t) \,, \\
    D_t v_N(x,t) + v_N(x,t)J_N(x) = DF_0(\gamma_0(x,t)) v_N(x,t) \,, 
\end{aligned}
$$
and, hence,
$$
\begin{aligned}
    D V_M(x,\phi)r_0(x,\phi) + V_M(x,\phi)J_M(x) &=
\omega(x) D_{\phi}V_M(x,\phi)+ V_M(x,\phi)J_M(x)
= DF_0(\gamma_0(x,\omega^{-1}\phi)) V_M(x,\phi) \,,\\
D V_N(x,\phi)r_0(x,\phi) + V_N(x,\phi)J_N(x) &=
\omega(x) D_{\phi}V_N(x,\phi)+ V_N(x,\phi)J_N(x)
= DF_0(\gamma_0(x,\omega^{-1}\phi)) V_N(x,\phi) \,.
\end{aligned}
$$

\subsection{A further splitting of the tangent bundle $TM_0=TR_0 \oplus TS_0$}\label{sec:subsplitting}

The tangent space $T_z\mathbb{R}^n$ at any point $z=\gamma_0(x,t)$ on $M_0$ can be split further to
$$
T_z\mathbb{R}^n = T_z M_0 \oplus {\cal N}_z = T_z R_0 \oplus T_z S_0 \oplus {\cal N}_z\,;
$$
see Figure~\ref{fig:embed-averaging}.
\begin{itemize}
    \item \textbf{Basis of $T_z R_0$ (tangential to the periodic orbit):}
    This space is one-dimensional and spanned by the velocity vector of the periodic orbit:
    $$
    v_R(x,t) = \frac{\partial}{\partial t}\gamma_0(x,t) = D_t\gamma_0(x,t) =F_0(\gamma_0(x,t))\,.
    $$
    This vector $v_R(x,t)$ is $\tau(x)$-periodic in $t$. We also define it in the $(x,\phi)$ coordinates:
    $$
    V_R(x,\phi):=v_R(x,\omega^{-1}\phi)=F_0(\Gamma_0(x,\phi))=
    D\Gamma_0(x,\phi)r_0(x,\phi)=D_\phi\Gamma_0(x,\phi)\,\omega(x)\,.
    $$
    Note that $D_tv_R(x,t)=DF_0(\gamma_0(x,t))v_R(x,t)$ and, hence,
    $$
    DV_R(x,\phi)r_0(x,\phi)=\omega(x) D_{\phi} V_{R}(x,\phi)=
    DF_0(\Gamma_0(x,\phi))D\Gamma_0(x,\phi)r_0(x,\phi)=DF_0(\Gamma_0(x,\phi))V_R(x,\phi)
    \,.
    $$
Finally, we can transport the basis vector $v_R(x,t_0)$ between the tangent spaces at different points along the periodic orbit via the transition matrix, i.e., 
$$
v_R(x,t)=\Phi(x,t,t_0)v_R(x,t_0)\,,
$$
respectively
$$
V_R(x,\phi)=\Phi(x,\omega^{-1}\phi,\omega^{-1}\phi_0)V_R(x,\phi_0)\,,
$$

\begin{remark}
    This basis eigenvector $v_R(x,t)$ respectively $V_R(x,\phi)$ is part of the generalised basis vectors in $v_M(x,t)$ respectively $v_M(x,\phi)$.
\end{remark}
    \item \textbf{Basis of $T_z S_0$ (phase-preserving tangent space):}
    The frequency  $\omega(x) = 2\pi/\tau(x)$ of the periodic orbits varies with $x\in U\subset\mathbb{R}^k$.
    We are looking for a `phase-preserving' basis, i.e., we define tangent directions with respect to $x_j$, $j=1,\ldots k$, that account for the changing frequency across the family, ensuring that a change in $x_j$ does not induce a phase slip relative to the unperturbed orbit.
    Recall, the phase of an orbit at time $t$ is $\phi(x,t) = \omega(x)t$. To define a phase-preserving derivative with respect to $x_j$, we implicitly allow $t$ to vary such that the phase $\phi$ remains constant. Taking the total derivative of $\phi$ with respect to $x_j$ gives:
    $$
    \frac{d\phi}{dx_j} = \frac{\partial \omega}{\partial x_j} t + \omega(x) \frac{\partial t}{\partial x_j} = D_{x_j}\omega(x) t+\omega(x)D_{x_j}t =
    0\,.
    $$
    From this, the required time shift to maintain constant phase is:
    $$
    \left. D_{x_j} t\right|_{\phi=\text{const}} =  
    - \frac{t}{\omega(x)} D_{x_j}\omega(x)= \frac{t}{\tau(x)} D_{x_j}\tau(x)\,.
    $$
    Using the chain rule, the phase-preserving basis vectors for $T_z S_0$ are defined as:
    $$
    \begin{aligned}
        v_{S_j}(x,t) & =  
    D_{x_j}\gamma_0(x,t) + D_{t}\gamma_0(x,t) \left. D_{x_j}t \right|_{\phi=\text{const}}\\
     & = D_{x_j}\gamma_0(x,t) + \frac{t}{\tau(x)} D_{x_j}\tau(x) D_{t}\gamma_0(x,t)\\
    & = D_{x_j}\gamma_0(x,t) - \frac{t}{\omega(x)} D_{x_j}\omega(x) D_{t}\gamma_0(x,t)\,,\qquad \mbox{for } j=1,\dots,k\,.
    \end{aligned}
$$
    These $k$ vectors $(v_{S_1},\ldots,v_{S_k})$ form a basis for $T_z S_0$ and are linearly independent of $v_R(x,t)$ and from each other. They are $\tau(x)$-periodic in $t$, i.e., $v_{S_j}(x,t)=v_{S_j}(x,t+\tau(x))$.  
    
    In $(x,\phi)$-coordinates, these `phase-preserving' base vectors are then simply
    $$
    \begin{aligned}
        V_{S_j}(x,\phi):=v_{S_j}(x,\omega^{-1}\phi) 
        &= D_{x_j}\Gamma_0(x,\phi)\,,
    \end{aligned}
    $$
which are $2\pi$-periodic in $\phi$ (since $\Gamma_0(x,\phi)$ is also $2\pi$-periodic in $\phi$).

    Note that 
    $$D_tv_{S_j}(x,t)=DF_0(\gamma_0(x,t))v_{S_j}(x,t)-\omega^{-1}(x)D_{x_j}\omega(x) v_{R}(x,t)$$ 
    and, hence,
    $$
    \begin{aligned}
        DV_{S_j}(x,\phi)r_0(x,\phi)
         =D_\phi V_{S_j}(x,\phi)\omega(x) &=
    DF_0(\Gamma_0(x,\phi))V_{S_j}-
    D_{x_j}\omega(x)D_\phi\Gamma_0(x,\phi)\\
    &=DF_0(\Gamma_0(x,\phi))V_{S_j}-
    \omega^{-1}(x) D_{x_j}\omega(x) V_R
    \,.
    \end{aligned}
    $$
\end{itemize}
%
Finally, let $v_S:=(v_{S_1},\ldots,v_{S_k})$ be the $n\times k$ matrix formed by these basis vectors. The $n\times (k+1)$ matrix $v_M:=(v_R,v_S)$ includes all basis vectors of $T_zM_0$. Similarly, we define matrices $V_S$ and $V_M$.

\begin{remark}
    The basis vectors in $v_S$ (respectively $V_S$) are, in general, not invariant under the action of the transition matrix, because of the phase correction, i.e., the algebraic and geometric multiplicity of the zero eigenvalue along $M_0$ is not the same. On the other hand, $TR_0$ itself is invariant, and so is the whole tangent space $TM_0=TR_0\oplus TS_0$. Thus, with respect to these basis vectors $v_M:=(v_R,v_S)$ respectively $V_M:=(V_R,V_S)$, the Jordan-block matrix $J_{M}$ is triangular.
    
    The matrix $V_M$ constructed here from geometric considerations is a specific choice of basis for the tangent bundle $TM_0$. It spans the same space as the basis $V_M$ derived from the Floquet analysis in section~\ref{sec:mono-and-normal-bundle}, though the basis vectors themselves may differ.
\end{remark}

\subsection{Oblique projection operators}\label{sec:proj-op}

Based on our bundle splitting, we define oblique projection operators $\PiR(x,t)$, $\PiS(x,t)$, and $\PiN(x,t)$, $(x,t)\in U\times\mathbb{S}^1$, such that $\PiR(x,t) + \PiS(x,t) + \PiN(x,t) = \mathbb{I}_n$. We also define $\PiM(x,t)=\PiR(x,t)+\PiS(x,t)$.
Since $T_z R_0$, $T_z M_0$ and $N_z$ ($z=\gamma_0(x,t)$) are invariant under the flow $\Phi(x,t,t_0)$, the projections 
$\PiM(x,t)$ onto $T_{\gamma(x,t)}M_0$, 
$\PiR(x,t)$ onto $T_{\gamma(x,t)}R_0$ and $\PiN(x,t)$ onto $N_{\gamma(x,t)}$ can be obtained from the initial projections 
$\PiM(x,t_0)$, 
$\PiR(x,t_0)$ and $\PiN(x,t_0)$ via 
$$\Pi_{i}(x,t) := \Phi(x,t,t_0) \Pi_{i}(x,t_0) \Phi(x,t,t_0)^{-1}\,,\qquad i \in \{R,M,N\}\,.$$
Hence, the projection $\PiS(x,t)$ onto $T_{\gamma(x,t)}S_0$ is then defined as 
$$\PiS(x,t):=\PiM(x,t)-\PiR(x,t)=\mathbb{I}_n-\PiN(x,t)-\PiR(x,t)\,.$$
All these time-dependent projection operators have the semi-group property
$$
\Pi_i(x,t)\Phi(x,t,s)=\Phi(x,t,s)\Pi_i(x,s)\,,\quad\forall t,s\in\mathbb{R}
$$
and they solve
$$
D_t\Pi_i(x,t)= A(x,t)\Pi_i(x,t)-\Pi_i(x,t)A(x,t)=:[A(x,t),\Pi_i(x,t)]\,.
$$
\begin{remark}
Let $\PiNs(x,t)$ and $\PiNu(x,t)$ be the projections onto the stable and unstable normal bundles at $z=\gamma_0(x,t)$.
Normal hyperbolicity implies
    for the stable normal bundle:
    $$||\Phi(x,t,s)\PiNs(x,s)|| \le K e^{-\lambda_s(t-s)} \quad \text{for } t \ge s\,,$$
    and for the unstable normal bundle:
    $$||\Phi(x,t,s)\PiNu(x,s)|| \le K e^{\lambda_u(t-s)} \quad \text{for } t \le s\,.$$
\end{remark}
%
%
\noindent
Next, we outline two (possible) methods for the construction of these oblique projection operators

\subsubsection{Constructing projections from a basis and its dual}\label{sec:projection1}

The most direct way to define the projection operators is to construct them from an explicit basis for the tangent and normal spaces.

\paragraph{1. Form a basis for the full space.}
At any point $z = \Gamma_0(x,\phi)$ on the manifold, we have already defined a basis for the tangent space $T_z M_0$ (the columns of the $n \times (k+1)$ matrix $V_M(x,\phi)$) and the normal space $\mathcal{N}_z$ (the columns of the $n \times (n-k-1)$ matrix $V_N(x,\phi)$). We can combine these into a single $n \times n$ matrix whose columns form a basis for the entire ambient space $\R^n$:
$$
W(x,\phi) = \begin{pmatrix} V_M(x,\phi) & 
V_N(x,\phi) \end{pmatrix}
$$
Since the tangent and normal bundles form a direct sum, this matrix $W$ is invertible.

\paragraph{2. Compute the dual basis.}
The rows of the inverse matrix $W^{-1}$ form the \textit{dual basis}. We partition this inverse matrix conformally with $W$:
$$
W(x,\phi)^{-1} = \begin{pmatrix} V_M^*(x,\phi) \\ 
V_N^*(x,\phi) \end{pmatrix}
$$
Here, $V_M^*$ is a $(k+1) \times n$ matrix whose rows are the dual basis vectors to the tangent space, and $V_N^*$ is an $(n-k-1) \times n$ matrix for the normal space, i.e., they satisfy: 
$$V_M^* V_M = \mathbb{I}_{k+1},\quad V_N^* V_N = \mathbb{I}_{n-k-1}, \quad V_M^* V_N = \mathbb{O}_{k+1,n-k-1},\quad V_N^* V_M = \mathbb{O}_{n-k-1,k+1}\,.$$

\paragraph{3. Construct the projectors.}
The oblique projection operators are then given by the outer product of the basis and dual basis vectors:
\begin{itemize}
    \item \textbf{Projection onto the tangent bundle $TM_0$:}
    $$ \Pi_M(x,\phi) = V_M(x,\phi) V_M^*(x,\phi) $$
    \item \textbf{Projection onto the normal bundle $\mathcal{N}$:}
    $$ \Pi_N(x,\phi) = V_N(x,\phi) V_N^*(x,\phi) $$
\end{itemize}
By construction, these operators satisfy $\Pi_M + \Pi_N = \mathbb{I}_n$, $\Pi_M^2 = \Pi_M$, $\Pi_N^2 = \Pi_N$, and $\Pi_M \Pi_N = \Pi_N \Pi_M = \mathbb{O}_n$. This method provides a direct and computationally feasible way to find the projectors once the basis vectors are known.

\paragraph{4. Construct sub-projectors for the tangent bundle.}
The same principle allows for the decomposition of the tangent space projector $\Pi_M$ into its constituent parts. Recall from Section~\ref{sec:subsplitting} that the basis for the tangent bundle, $V_M(x,\phi)$, is the concatenation of the basis for the flow direction, $V_R(x,\phi)$ (an $n \times 1$ matrix), and the basis for the phase-preserving directions, $V_S(x,\phi)$ (an $n \times k$ matrix):
$$
V_M(x,\phi) = \begin{pmatrix} V_R(x,\phi) & V_S(x,\phi) \end{pmatrix} \, .
$$
The dual basis matrix $V_M^*(x,\phi)$ can be partitioned conformally:
$$
V_M^*(x,\phi) = \begin{pmatrix} V_R^*(x,\phi) \\ V_S^*(x,\phi) \end{pmatrix}
$$
where $V_R^*$ is $1 \times n$ and $V_S^*$ is $k \times n$. By the properties of the dual basis, these partitions satisfy:
$$
\begin{pmatrix} V_R^* V_R & V_R^* V_S \\ V_S^* V_R & V_S^* V_S \end{pmatrix} = \begin{pmatrix} 1 & \mathbb{O}_{1,k} \\ \mathbb{O}_{k,1} & \mathbb{I}_k \end{pmatrix}
$$
This allows for the immediate construction of the sub-projectors:
\begin{itemize}
    \item \textbf{Projection onto the flow bundle $TR_0$:}
    $$ \Pi_R(x,\phi) = V_R(x,\phi) V_R^*(x,\phi) $$
    \item \textbf{Projection onto the phase-preserving bundle $TS_0$:}
    $$ \Pi_S(x,\phi) = V_S(x,\phi) V_S^*(x,\phi) $$
\end{itemize}
These operators are themselves projectors ($\Pi_R^2 = \Pi_R$, $\Pi_S^2 = \Pi_S$), are mutually annihilating ($\Pi_R \Pi_S = \Pi_S \Pi_R = \mathbb{O}_n$), and sum to the tangent space projector: $\Pi_R(x,\phi) + \Pi_S(x,\phi) = \Pi_M(x,\phi)$.

\subsubsection{Constructing projections from a Floquet-geometric basis}

Alternatively, the projectors can be constructed using a hybrid approach that combines the spectral properties of the Floquet matrix $B(x)$ with the geometric definitions of the tangent basis vectors. This method is particularly powerful as it uses each tool for what it defines best: Floquet theory to identify the invariant normal bundle, and geometry to provide the physically meaningful split of the tangent bundle.

The procedure is as follows:
\begin{enumerate}
    \item \textbf{Identify the normal space using Floquet theory.} From the Floquet matrix $B(x)$, compute its semi-simple part, $C(x)$. The normal bundle $\mathcal{N}$ corresponds to the image of this matrix. Therefore, compute a basis for the normal bundle as the columns of an $n \times (n-k-1)$ matrix:
    $$ V_N(x) = \text{basis}(\text{Im}(C(x))) $$

    \item \textbf{Identify the tangent basis using geometry.} From the geometric definitions in Section 2.4, construct the basis for the tangent bundle, which is already split into the flow and phase-preserving directions. These are the columns of the $n \times (k+1)$ matrix $V_M$:
    $$ V_M(x,\phi) = \begin{pmatrix} V_R(x,\phi) & V_S(x,\phi) \end{pmatrix} $$
    
    \item \textbf{Form a full basis for $\mathbb{R}^n$.} Combine the bases for the tangent and normal bundles to form a single, invertible $n \times n$ matrix whose columns span the entire ambient space. Note that the geometric part $V_M$ is $\phi$-dependent, while the normal part $V_N$ is not. We evaluate at a reference phase $\phi_0$:
    $$ W(x,\phi_0) = \begin{pmatrix} V_M(x,\phi_0) & V_N(x) \end{pmatrix} = \begin{pmatrix} V_R(x,\phi_0) & V_S(x,\phi_0) & V_N(x) \end{pmatrix} $$
    
    \item \textbf{Compute the dual basis.} Invert the full basis matrix $W$ to get the dual basis, which is given by its rows. We partition the inverse conformally with $W$:
    $$ W(x,\phi_0)^{-1} = \begin{pmatrix} V_R^*(x,\phi_0) \\ V_S^*(x,\phi_0) \\ V_N^*(x) \end{pmatrix} $$
    
    \item \textbf{Construct all oblique projectors.} Using the basis and dual basis vectors, the complete set of projectors at the reference phase $\phi_0$ is given by the outer products:
    \begin{align*}
        \Pi_R(x,\phi_0) &= V_R(x,\phi_0) V_R^*(x,\phi_0) \\
        \Pi_S(x,\phi_0) &= V_S(x,\phi_0) V_S^*(x,\phi_0) \\
        \Pi_N(x,\phi_0) &= V_N(x) V_N^*(x)
    \end{align*}
    These projectors can then be propagated to any phase $\phi$ along the periodic orbits by conjugation with the flow operator.
\end{enumerate}

\section{The parametrisation method}\label{sec:para}

Section \ref{sec:problem_setup} established the geometric setup under the unperturbed flow ($\eps=0$), namely the existence of a normally hyperbolic manifold $\Mzero$ foliated by periodic orbits. We now seek to determine the fate of this manifold and the dynamics upon it for $\eps>0$. Fenichel's theory \cite{fenichel1979} guarantees that a nearby invariant manifold, $\Meps$, persists (Theorem~\ref{thm:Fenichel}). The parametrisation method provides a constructive and systematic framework not only to prove this persistence but also to compute both the embedding of $\Meps$ and the reduced, averaged flow on it as formal power series in $\eps$.

\begin{assumption}
    We assume the embedding of this manifold $M_\epsilon$ is given as a series expansion in $\epsilon$:
\begin{equation}
    \Gamma(x,\phi,\epsilon) = \sum_{i=0}^m \epsilon^i \Gamma_i(x,\phi)\,,
\end{equation}
with $m \ge 1$.
The corresponding vector field $r(x,\phi,\epsilon)$ that describes the dynamics on the perturbed manifold $M_\epsilon$ (when expressed in the $(x,\phi)$ coordinates inherited from $M_0$) is also expanded in a series:
\begin{equation}
    r(x,\phi,\epsilon) = \sum_{i=0}^m \epsilon^i r_i(x,\phi)    \, .
\end{equation}
\end{assumption}
\noindent
Recall, the original system dynamics \eqref{eq:intro_system} are given by 
$$z' = F(z,\epsilon) = \sum_{i=0}^m \epsilon^iF_i(z)\,.$$

\subsection{The invariance condition}

The core invariance condition states that the velocity vector obtained by differentiating the embedding $\Gamma(x,\phi,\epsilon)$ with respect to time (using the intrinsic flow $r(x,\phi,\epsilon)$ on the manifold) must be equal to the original system's vector field $F$ evaluated on the manifold $M_\epsilon$:
\begin{equation}\label{eq:invariance-full}
    D\Gamma(x,\phi,\eps)r(x,\phi,\eps)=D_x \Gamma(x,\phi,\eps) r_x(x,\phi,\eps) + D_\phi \Gamma(x,\phi,\eps) r_\phi(x,\phi,\eps) = F(\Gamma(x,\phi,\eps),\eps)
\end{equation}
Here, $D_x \Gamma$ is the Jacobian with respect to $x$ (an $n \times k$ matrix), $D_\phi \Gamma$ is the Jacobian with respect to $\phi$ (an $n \times 1$ vector), and $r(x,\phi,\eps) = \begin{pmatrix} r_x(x,\phi,\epsilon) \\ r_\phi(x,\phi,\eps) \end{pmatrix}$ is the $(k+1)$-dimensional vector field on $M_\eps$ in the local coordinate chart $U\times \mathbb{S}^1$.

\subsubsection*{Solving order by order using projections}

We substitute the above series expansions for $\Gamma(x,\phi,\eps)$, $r(x,\phi,\eps)$, and $F(\Gamma(x,\phi,\eps),\eps)$ into the invariance condition \eqref{eq:invariance-full}. Collecting terms by powers of $\eps$ then yields the following:\\

\noindent 
\textbf{Order $\eps^0$:}
$$
D \Gamma_0(x,\phi) r_{0}(x,\phi) =
D_x \Gamma_0(x,\phi) r_{x,0}(x,\phi) + D_\phi \Gamma_0(x,\phi) r_{\phi,0}(x,\phi) = F_0(\Gamma_0(x,\phi))\,.
$$
Substituting $r_{x,0}=0$ and $r_{\phi,0}=\omega(x)$ gives 
$$
D_\phi \Gamma_0(x,\phi)\, \omega(x) = F_0(\Gamma_0(x,\phi))\,.$$
This equation is satisfied by the definition of $\Gamma_0$ as an embedding of the flow of $F_0$; see \eqref{eq:conjugacy}.\\

\noindent 
\textbf{Order $\eps^1$:}
$$
D \Gamma_0(x,\phi) r_{1}(x,\phi) + D \Gamma_1(x,\phi) r_{0}(x,\phi) =
DF_0(\Gamma_0) \Gamma_1 + F_1(\Gamma_0)\,.$$
Rearranging terms to group $\Gamma_1$ and $r_1$ we find 
$$
(D \Gamma_1(x,\phi) r_{0}(x,\phi) - DF_0(\Gamma_0) \Gamma_1) + D \Gamma_0(x,\phi) r_{1}(x,\phi)= F_1(\Gamma_0)\,.
$$
Substituting $r_{x,0}=0$ and $r_{\phi,0}=\omega(x)$ then gives
$$
\left( D_\phi \Gamma_1(x,\phi) \omega(x) - DF_0(\Gamma_0(x,\phi)) \Gamma_1(x,\phi) \right) + D\Gamma_0(x,\phi) r_{1}(x,\phi) = F_1(\Gamma_0(x,\phi))\,.
$$
Let $\mathcal{L}_0[\cdot]$ be the linear operator acting on $\Gamma_1$, representing its evolution along the unperturbed flow: 
$$
\mathcal{L}_0[\Gamma_1](x,\phi) := 
D \Gamma_1(x,\phi) r_{0}(x,\phi) - DF_0(\Gamma_0(x,\phi)) \Gamma_1(x,\phi) =
D_\phi \Gamma_1(x,\phi) \omega(x) - DF_0(\Gamma_0(x,\phi)) \Gamma_1(x,\phi)\,.
$$
The equation becomes:
$$
\mathcal{L}_0[\Gamma_1] + D\Gamma_0 r_{1} = F_1(\Gamma_0)\,.
$$

\noindent
\textbf{Order $\eps^j$:}
This iterative process extends to higher orders ($\eps^2, \eps^3, \dots$). At each order $j$ we obtain
$$
\mathcal{L}_0[\Gamma_j] + D\Gamma_0 r_{j} = \mathcal{G}_j(x,\phi)\,,
$$
where $\mathcal{G}_j$ is the known inhomogeneous term at order $j$, which depends on the functions $\Gamma_l, r_l$ from previous orders ($l<j$) and on the term $F_j(\Gamma_0)$. Note that $\mathcal{G}_1=F_1$.

\subsection{The homological equation and its solution}

The goal is to solve the dynamics on $M_\epsilon$ via the parametrisation method. As shown, after expanding the invariance equation, we must solve a sequence of linear problems. For each order $j \ge 1$, this is the {\em homological equation}:
\begin{equation} \label{eq:homological}
    \mathcal{L}_0[\Gamma_j] + D\Gamma_0(x,\phi) r_j = \mathcal{G}_j(x,\phi)\,,
\end{equation}
where the linear operator $\mathcal{L}_0$ is defined as
\begin{equation}
    \mathcal{L}_0[\cdot] := \omega(x)\pderiv{}{\phi}[\cdot] - DF_0(\Gamma_0(x,\phi))[\cdot]\,.
\end{equation}
At the start of step $j$, the inhomogeneity $\mathcal{G}_j$ is known from previous orders. The unknowns are the embedding correction $\Gamma_j(x,\phi)$ and the vector field correction $r_j(x,\phi)$.

\subsubsection{The solvability condition}

A direct solution for $\Gamma_j$ by inverting $\mathcal{L}_0$ is not possible. The operator $\mathcal{L}_0$ is singular because it has a non-trivial null space, which contains the tangent vectors of the manifold $M_0$. 
%
The existence of a solution is therefore governed by the {\it Fredholm alternative}. It states that for the equation $\mathcal{L}_0[\Gamma_j] = \mathcal{G}_j - D\Gamma_0 r_j$, a solution $\Gamma_j$ exists if and only if the right-hand side is orthogonal to the null space of the adjoint operator, $\ker(\mathcal{L}_0^*)$.
The adjoint operator $\mathcal{L}_0^*$ is defined with respect to the $L^2$ inner product over the periodic domain of the fast angle $\phi$, i.e.,
$$ \langle u, v \rangle_\phi := \frac{1}{2\pi} \int_0^{2\pi} u(x,\phi)^T v(x,\phi) \, d\phi\,. $$
Using integration by parts and the $2\pi$-periodicity of the functions in $\phi$, the adjoint operator is found to be:
$$ 
\mathcal{L}_0^*[\cdot] = -\omega(x)\pderiv{}{\phi}[\cdot] - DF_0(\Gamma_0(x,\phi))^T [\cdot]\,. 
$$
The null space of $\mathcal{L}_0^*$ is spanned by a basis of $k+1$ linearly independent solutions $\psi_i(x,\phi)$ for $i=1, \dots, k+1$, to the homogeneous adjoint equation $\mathcal{L}_0^*[\psi] = 0$. These basis functions must also be $2\pi$-periodic in $\phi$. 

\begin{definition}[solvability condition]
    The right hand side $\mathcal{G}_j - D\Gamma_0 r_j$ must be orthogonal to all basis vectors $\psi_m$:
\begin{equation} \label{eq:solvability_condition}
    \left\langle \mathcal{G}_j(x,\phi) - D\Gamma_0(x,\phi) r_j(x,\phi), \psi_i(x,\phi) \right\rangle_\phi = 0 \quad \text{for } i=1, \dots, k+1.
\end{equation}
\end{definition}
\noindent
This provides $k+1$ constraint equations. These are sufficient to uniquely determine the slow dynamics of the $x$-components of $r_j$, but leave the $\phi$-dependent part of the flow undetermined.

\subsubsection{Achieving a maximally reduced flow}

The undetermined $\phi$-dependent part of the flow $r_j$ can be resolved by using the other freedom in the system: the non-uniqueness of the manifold correction $\Gamma_j$. The freedom to add any tangential component to $\Gamma_j$ (any element from the null space of $\mathcal{L}_0$) is precisely the tool needed to resolve the ambiguity in $r_j$. We can use this freedom of parameterisation to systematically absorb the entire $\phi$-dependent part of the flow, justifying the simplification of solving for a flow that is independent of $\phi$.
\begin{definition}\label{def:max_reduced}
    A {\em maximally reduced flow} is one where the vector field $r(x,\phi,\eps)$ is independent of the fast angle $\phi$, i.e.,
\begin{equation}
    \begin{aligned}
    x' &= r_x(x, \eps) \\
    \phi' &= r_\phi(x, \eps)
\end{aligned}
\end{equation}
\end{definition}
\noindent
Let us decompose the homological equation \eqref{eq:homological} into its tangential and normal parts by applying projectors $\Pi_M$ and $\Pi_N$. The tangential part is:
\begin{equation} \label{eq:tangent_solve_full}
    \Pi_M(\mathcal{L}_0[\Gamma_j]) + D\Gamma_0 r_j = \Pi_M(\mathcal{G}_j)\,.
\end{equation}
The term $\Pi_M(\mathcal{L}_0[\Gamma_j])$ can be simplified to $\mathcal{L}_0[\Gamma_j^M] + \Pi_M(\mathcal{L}_0[\Gamma_j^N])$. Here, we focus on the freedom in choosing $\Gamma_j^M$.\footnote{The second term $\Pi_M(\mathcal{L}_0[\Gamma_j^N])$ contains no freedom and will be determined in the final step.}
We decompose the flow correction $r_j$ and the projected inhomogeneity $\Pi_M(\mathcal{G}_j)$ into their $\phi$-averages (denoted by a bar) and their purely oscillatory parts (denoted by a tilde):
$$ r_j(x,\phi) = \bar{r}_j(x) + \tilde{r}_j(x,\phi)\,, \qquad \Pi_M(\mathcal{G}_j) = \overline{\Pi_M \mathcal{G}_j}(x) + \widetilde{\Pi_M \mathcal{G}_j}(x,\phi)\,. $$
The tangential equation involves the choice of $\Gamma_j^M$ and the unknown $\tilde{r}_j$. We can now make the choice by defining the tangential correction $\Gamma_j^M$ to be the unique periodic solution to an equation that absorbs all oscillatory terms over which we have freedom:
\begin{equation} \label{eq:choice_of_gamma_T}
    \mathcal{L}_0[\Gamma_j^M] := \widetilde{\Pi_M \mathcal{G}_j}(x,\phi) - D\Gamma_0(x,\phi) \tilde{r}_j(x,\phi) - \widetilde{\Pi_M(\mathcal{L}_0[\Gamma_j^N])}\,. 
\end{equation}
This choice is equivalent to performing a $\phi$-dependent parameterisation of the manifold $z=\Gamma(x,\phi,\eps)$. It is guaranteed to have a solution for $\Gamma_j^M$ because the right-hand side of \eqref{eq:choice_of_gamma_T} is purely oscillatory by construction. It ensures that all oscillatory terms in the tangential equation cancel out, leaving a simple algebraic relation for the averaged parts:

\begin{equation}
    \overline{D\Gamma_0 \bar{r}_j} = \overline{\Pi_M \mathcal{G}_j} - \overline{\Pi_M(\mathcal{L}_0[\Gamma_j^N])}\,.
\end{equation}
This confirms that we can find a maximally reduced flow correction, $r_j(x) := \bar{r}_j(x)$, because any $\phi$-dependence could be absorbed into the parameterization of $M_\epsilon$. 

The constructive approach above is equivalent to enforcing a fundamental orthogonality principle. This condition provides a rule that a (unique) solution must satisfy and forms the basis of a computational procedure.
\begin{definition}
To ensure a unique solution for the manifold correction $\Gamma_j$ that is consistent with a maximally reduced flow, we require that the correction is orthogonal to the basis of the adjoint null space:
\begin{equation} \label{eq:gauge_condition1}
    \langle \Gamma_j(x,\phi), \psi_i(x,\phi) \rangle_\phi = 0 \quad \text{for } i=1, \dots, k+1.
\end{equation}
\end{definition}
\noindent
This specific choice systematically enforces the absorption of the oscillatory flow components described in the constructive approach.


\subsection{The iterative solution procedure} \label{subsec:solnprocedure}

The homological equation \eqref{eq:homological} is solved for each order $j \ge 1$ using a sequential, three-step process. 

\subsubsection*{Step 1: Determine the averaged dynamics $r_j(x)$}

The solvability condition is used to \emph{derive} a vector field correction $r_j$ that is independent of $\phi$, i.e., a solution $r_j(x,\phi) = r_j(x)$.
Substituting $r_j(x)$ into the solvability condition \eqref{eq:solvability_condition} and rearranging gives:
$$ \left\langle D\Gamma_0(x,\phi) r_j(x), \psi_i(x,\phi) \right\rangle_\phi = \left\langle \mathcal{G}_j(x,\phi), \psi_i(x,\phi) \right\rangle_\phi\,. $$
Since $r_j(x)$ is independent of the integration variable $\phi$, it can be factored out of the inner product:
$$ \left( \frac{1}{2\pi} \int_0^{2\pi} D\Gamma_0(x,\phi)^T \psi_i(x,\phi) \, d\phi \right)^T r_j(x) = \frac{1}{2\pi} \int_0^{2\pi} \mathcal{G}_j(x,\phi)^T \psi_i(x,\phi) \, d\phi \,.$$
This gives a system of $k+1$ linear algebraic equations for the $k+1$ unknown components of $r_j(x)$: 
$$ K(x) r_j(x) = b(x)\,, $$
where $K(x)$ is a $(k+1) \times (k+1)$ matrix whose entries are the averaged projections of the tangent vectors onto the adjoint basis, i.e., the $i$-th row of $K(x)$ is
$$
[K(x)]_i=\left( \frac{1}{2\pi} \int_0^{2\pi} D\Gamma_0(x,\phi)^T \psi_i(x,\phi) \, d\phi \right)^T,\quad i=1,\ldots k+1\,,
$$
and $b(x)$ is a $(k+1)$-dimensional vector from the averaged projection of the known inhomogeneity $\mathcal{G}_j$, i.e.,
$$
b_i(x) = \frac{1}{2\pi} \int_0^{2\pi} \mathcal{G}_j(x,\phi)^T \psi_i(x,\phi) \, d\phi\,,\quad i=1,\ldots k+1\,.
$$
The assumption that $M_0$ is a {normally hyperbolic manifold guarantees that the tangent and normal dynamics are spectrally separated, which in turn ensures that this matrix $K(x)$ is invertible. Therefore, this can be uniquely solved for the slow vector field correction:
\begin{equation}
    r_j(x) = K(x)^{-1} b(x)\,.
\end{equation}
%
%
This step determines the dynamics on the manifold at order $\epsilon^j$ using only the known inhomogeneity $\mathcal{G}_j$. Computationally, the application of the solvability condition via inner products with the adjoint basis vectors $\psi_m$ is equivalent to projecting the homological equation onto the tangent bundle $TM_0$ with $\PiM$ and averaging over $\phi$.



\subsubsection*{Step 2: Determine the normal correction $\Gamma_j^N(x,\phi)$}

The normal part is found by projecting the homological equation onto the normal bundle:
\begin{equation} \label{eq:normal_ode_final}
        \Pi_N (\mathcal{L}_0[\Gamma_j^N]) = \Pi_N (\mathcal{G}_j(x,\phi))\,.
    \end{equation}
Note that the terms $\mathcal{L}_0[\Gamma_j^M]$ and $D\Gamma_0 r_j$ are purely tangential and vanish under the projection $\Pi_N$.\footnote{The tangent bundle $TM_0$ is an invariant subspace for the linearized flow $DF_0(\Gamma_0)$. As a result, the operator $\mathcal{L}_0[\cdot] = \omega(x)\pderiv{}{\phi}[\cdot] - DF_0(\Gamma_0)[\cdot]$ maps tangent vectors to tangent vectors.}
The assumption of normal hyperbolicity ensures that the operator $\mathcal{L}_0$, when restricted to the normal bundle, is invertible, i.e., the operator $\mathcal{L}_0^N := \PiN \mathcal{L}_0 \PiN$ has a bounded inverse, and, hence,  
\eqref{eq:normal_ode_final} has a unique $2\pi$-periodic solution.\\

\noindent
One way to solve for the unique periodic solution is to use the variation of parameters formula. This requires the transition matrix for the dynamics projected onto the normal bundle, $\Phi_N(x, \phi, \sigma)$. This matrix is derived from the original transition matrix $\Phi(x,t,s)$ by changing variables to phase $\phi$ (where $t=\omega(x)^{-1}\phi$) and projecting:
$$ 
\Phi_N(x, \phi, \sigma) := \PiN(x,\phi) \, \Phi (x, \omega(x)^{-1}\phi, \omega(x)^{-1}\sigma) \, \PiN(x,\sigma) \,.
$$
The solution for an initial condition $\Gamma_j^N(x, \phi_0)$ given at an arbitrary initial phase $\phi_0$ is then given by the integral:
\begin{equation}
    \Gamma_j^N(x,\phi) = \Phi_N(x, \phi, \phi_0) \Gamma_j^N(x, \phi_0) + \frac{1}{\omega(x)}\int_{\phi_0}^\phi \Phi_N\left(x, \phi, \sigma\right) f_N(x,\sigma) d\sigma\,,
\end{equation}
where $f_N$ is the known right-hand side of \eqref{eq:normal_ode_final} and $\sigma$ is the integration variable for the phase. The initial condition is determined uniquely by enforcing the periodicity condition $\Gamma_j^N(x, 2\pi+\phi_0) = \Gamma_j^N(x, \phi_0)$, i.e.,
$$ 
\Gamma_j^N(x, 2\pi+\phi_0) =
M_N(x,\phi_0)\Gamma_j^N(x, \phi_0) + \frac{1}{\omega(x)}\int_{\phi_0}^{2\pi+\phi_0} \Phi_N\left(x, 2\pi+\phi_0, \sigma\right) f_N(x,\sigma) d\sigma =
\Gamma_j^N(x, \phi_0)\,,
$$
where 
$$ 
M_N(x, \phi_0) := \PiN(x,\phi_0) \, M (x,\phi_0) \, \PiN(x,\phi_0)
$$
is the projection of the full monodromy matrix $M(x,\phi_0)$ onto the normal bundle ${\cal N}$. This gives
$$ 
\Gamma_j^N(x, \phi_0) = \omega(x)^{-1}
(\mathbb{I}_n -  
M_N(x,\phi_0))^{-1}
\int_{\phi_0}^{2\pi +\phi_0} \Phi_N\left(x, 2\pi +\phi_0, \sigma\right) f_N(x,\sigma) d\sigma \,.
$$
\subsubsection*{Step 3: Determine the tangential correction $\Gamma_j^M(x,\phi)$}

A solution for $\Gamma_j$ is now guaranteed to exist, but it is not unique because any element from $\ker(\mathcal{L}_0)$ can be added. To obtain a unique solution, we must impose a final condition. This condition is what computationally enforces the reparametrisation that allows for a maximally reduced flow. The standard choice is to require the full correction $\Gamma_j$ to be orthogonal to the adjoint null space:
\begin{equation} \label{eq:gauge_condition}
    \langle \Gamma_j(x,\phi), \psi_i(x,\phi) \rangle_\phi = 0 \quad \text{for } i=1, \dots, k+1.
\end{equation}
Since $\langle \Gamma_j^N, \psi_i \rangle_\phi$ is already determined from Step 2, this condition uniquely fixes the remaining tangential component $\Gamma_j^M$ by requiring:
$$
\langle \Gamma_j^M(x,\phi), \psi_i(x,\phi) \rangle_\phi = - \langle \Gamma_j^N(x,\phi), \psi_i(x,\phi) \rangle_\phi\,,
$$
or, equivalently,
\begin{equation} \label{eq:gauge_condition2}
    \langle \Gamma_j(x,\phi), \psi_i(x,\phi) \rangle_\phi =
    \frac{1}{2\pi} \int_0^{2\pi} \Gamma_j(x,\phi)^T \psi_i(x,\phi)\, d\phi
    = 0 \quad \text{for } i=1, \dots, k+1.
\end{equation}
\begin{remark}
This orthogonality condition is not in conflict with the choice made in \eqref{eq:choice_of_gamma_T}. Rather, it is the specific condition that, when applied to the full homological equation, ensures that the tangential correction $\Gamma_j^M$ has the correct form to absorb the oscillatory terms, thus making the maximally reduced flow ansatz consistent.
\end{remark}
%
%
\noindent
This completes the procedure for order $j$. We can now proceed to order $j+1$, constructing $\mathcal{G}_{j+1}$ and repeating the process.

\begin{remark}
The procedure described uses a specific choice of reparameterization to achieve a maximally reduced flow. Other choices are possible, which would lead to a flow $r_j$ that depends on $\phi$ but might simplify the expression for the manifold correction $\Gamma_j$. The choice presented here is standard as it simplifies the slow dynamics as much as possible.
\end{remark}

\section{Main results and the averaged flow}\label{sec:results}

The parametrisation method detailed in Section \ref{sec:para} provides a constructive proof of the existence of a coordinate system in which the dynamics on the perturbed manifold $M_\eps$ take a simplified, averaged form. This section formalises that result.

\begin{theorem}[Existence of a skew-product flow]\label{thm:skew_product}
Consider system \eqref{eq:intro_system} under Assumptions \ref{ass:unperturbed} and \ref{ass:hyperbolicity}. 
There exist smooth functions $r_x: U \times \R \to \R^k$ and $r_\phi: U \times \R \to \R$ of the form
$$
\begin{aligned}
    r_x(x,\eps) &= \sum_{i=1}^m \eps^i r_{x,i}(x)  \,, \\
    r_\phi(x,\eps) &= \omega(x) + \sum_{i=1}^m \eps^i r_{\phi,i}(x)  \,,
\end{aligned}
$$
and an embedding $\Gamma: U \times \Sone \times \R \to \R^n$ of the form
$$
\Gamma(x,\phi,\eps) = \Gamma_0(x,\phi) + \sum_{i=1}^m \eps^i \Gamma_i(x,\phi)  
$$
with the following property: if $(x(t), \phi(t))$ is a solution to the skew-product system
\begin{equation}\label{eq:skew_product}
    \begin{aligned}
        x' & = r_x(x,\eps)\,,\\
        \phi' & = r_\phi(x,\eps)\,,
    \end{aligned}
\end{equation}
then $z(t) := \Gamma(x(t), \phi(t), \eps)$ is an approximate solution to the original system \eqref{eq:intro_system}, i.e., it satisfies
$$ z' = F(z, \eps) + \mathcal{O}(\varepsilon^{m+1})\,. $$
\end{theorem}

\begin{remark}
The procedure in Section \ref{subsec:solnprocedure} constructs precisely the series coefficients $r_{x,i}(x)$, $r_{\phi,i}(x)$, and $\Gamma_i(x,\phi)$ that fulfill this theorem under the `maximally reduced flow' choice, Definition~\ref{def:max_reduced}, where $r_{x,i}$ and $r_{\phi,i}$ are independent of $\phi$.
\end{remark}
\noindent
As a consequence, we can describe the slow dynamics by a flow on a representative cross-section of the manifold $M_0$.
\begin{definition}[Section of the manifold]
 The set
 $$ S_{0} := \Gamma_{0}(U \times \{0\}) = \{ \Gamma_{0}(x,0) \mid x \in U \} \subset M_{0} $$
 denotes a {\em section} of the  manifold $M_{0}$. It is a $k$-dimensional submanifold of $\mathbb{R}^n$. This section intersects each periodic orbit $\gamma_0(x, \cdot)$ exactly once. 
For any $y = \Gamma_0(x,0) \in S_0$, we denote the periodic orbit passing through this point as 
$$\tilde{\gamma}_y := \{ \Gamma_0(x,\phi) \mid \phi \in \Sone \}\,.$$
\end{definition}
\begin{figure}[t]
    \centering
    \includegraphics[width=11.0cm]{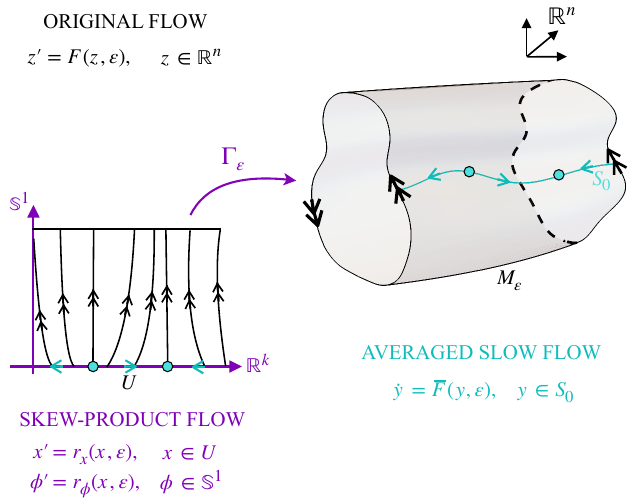}
    \caption{ A sketch of a normally hyperbolic manifold $M_{\varepsilon}$  embedded by $\Gamma_{\varepsilon}$. The slow drift along the unperturbed periodic orbits is governed by  the averaged slow flow along the unperturbed section $S_{0}$ and clearly visible in the local coordinate chart $U\times\mathbb{S}^1$. 
    }
    \label{fig:embed-averaging-new}
\end{figure}
\noindent  The next corollary allows us to interpret the slow drift on $M_\eps$ along the unperturbed periodic orbits in $M_0$ as an {\em averaged slow flow} on the section $S_{0}$, and provides the equations of motion governing this flow. The corollary forms a  higher-order and coordinate-independent generalisation of the classical Pontryagin-Rodygin slow-fast averaging theorem, see e.g., \cite{PontryaginRodygin1960}. We illustrate it schematically in Figure \ref{fig:embed-averaging-new}. 
\begin{corollary}[Averaged flow on the section]\label{cor:avg_flow}
   There exists a vector field $\overline{F}$ on the section $S_0$ of the form
   $$
   \overline{F}(y,\eps) = \sum_{j=1}^m \eps^j \overline{F}_j(y) + \mathcal{O}(\eps^{m+1})
   $$
   (i.e., $\overline{F}$ is tangent to $S_0$) with the following property: if $z(t)$ is a trajectory on the full manifold $M_\eps$ and $y(t)$ is an integral curve of the averaged system $y'=\overline{F}(y,\eps)$ with $\|z(0) - y(0)\| = \mathcal{O}(\eps)$, then
   $$
   \mathrm{dist}(z(t), \tilde{\gamma}_{y(t)}) = \mathcal{O}(\varepsilon) \quad \text{for} \quad |t| \leq \frac{C}{\varepsilon}.
   $$
\end{corollary}

\begin{proof}[Proof Sketch]
The averaged vector field $\overline{F}$ on $S_0$ is defined by pushing forward the slow vector field $r_x(x,\eps)$ from the parameter domain $U$ onto $S_0$ via the embedding $y = \Gamma_0(x,0)$:
$$ \overline{F}(y,\eps) := D_x\Gamma_0(x,0) \cdot r_x(x,\eps) \quad \text{for } y = \Gamma_0(x,0)\,. $$
Let $z(t) = \Gamma(\xi(t), \theta(t), \eps)$ be an exact trajectory on $M_\eps$, where $(\xi(t), \theta(t))$ solve the true dynamics on the manifold. Let $(x(t), \phi(t))$ solve the truncated skew-product system \eqref{eq:skew_product}. The difference between the ODEs for $\xi(t)$ and $x(t)$ is $\mathcal{O}(\eps^{m+1})$. Assuming nearby initial conditions, $\| \xi(0) - x(0) \| = \mathcal{O}(\eps)$, Gronwall's Lemma implies that
$$ \| \xi(t) - x(t) \| = \mathcal{O}(\eps) \quad \text{for } |t| \leq \frac{C}{\eps}\,. $$
Let $y(t) = \Gamma_0(x(t),0)$ be the solution on the section. The distance can be estimated by the triangle inequality:
\begin{align*}
    \mathrm{dist}(z(t), \tilde{\Gamma}_{y(t)}) &\leq \| \Gamma(\xi(t),\theta(t),\eps) - \Gamma_0(x(t),\theta(t)) \| \\
    &\leq \| \Gamma(\xi(t),\theta(t),\eps) - \Gamma_0(\xi(t),\theta(t)) \| + \| \Gamma_0(\xi(t),\theta(t)) - \Gamma_0(x(t),\theta(t)) \| \\
    &\leq \mathcal{O}(\eps) + L \| \xi(t) - x(t) \| = \mathcal{O}(\eps)\,,
\end{align*}
where the first term is $\mathcal{O}(\eps)$ because $M_\eps$ is $\mathcal{O}(\eps)$-close to $M_0$, and the second term is bounded by the Lipschitz constant $L$ of $\Gamma_0$ and the $\mathcal{O}(\varepsilon)$ estimate for $\| \xi(t) - x(t) \|$ derived from Gronwall's Lemma.
\end{proof}

\subsection{Relationship to averaging theory}

The parametrisation method provides a rigorous, geometric foundation for the results of classical averaging theory, such as the Pontryagin-Rodygin theorem \cite{PontryaginRodygin1960}. The derived maximally reduced flow $x' = r_x(x,\epsilon)$ is not merely a formal average but represents the true, slow dynamics on the persistent invariant manifold $M_\epsilon$.

\paragraph{Averaging on the correct manifold.}
Classical averaging theorems typically prescribe computing the average of the perturbation vector field along the known, unperturbed orbits on $M_0$. This provides a first-order approximation to the slow dynamics, and the leading-order term of our reduced flow, $r_1(x)$, recovers this classical result. The parametrisation method is more powerful because it systematically accounts for the deformation of the manifold itself. 
Thus, the method computes the average dynamics on the perturbed manifold $M_\epsilon$, leading to a more accurate and rigorously justified result beyond the leading order.

\paragraph{Accuracy and timescale of validity.}
The key advantage of this geometric approach lies in the improved \textit{accuracy} of the reduced model. The shadowing result holds on the classical averaging timescale of $t \sim O(1/\varepsilon)$; more refined estimates can even show this validity extends to the slightly longer timescale of $t \sim O(|\ln\varepsilon|/\varepsilon)$. The parametrisation method, however, provides a better approximation within that timescale. If the reduced system is computed up to order $k$, the difference between its solution and the true projected trajectory is of order $O(\epsilon^{k+1})$. 
By including higher-order terms, we create a more accurate model. This allows for improved predictions of the slow drift, as well as the location and stability of emergent structures within the slow flow, such as its equilibria or periodic orbits.

\section{Analytical example: the bent ellipsoid} \label{sec:ellipsoid}
We now apply our techniques to a model system constructed to be analytically tractable. The governing singularly perturbed system (in general form) is
\[ \frac{dz}{dt} = F_0(z) + \eps F_1(z). \]
To ensure we can compute everything analytically, we build the unperturbed vector field $F_0(z)$ with a specific geometric structure in mind.
%
First, we define the manifold of periodic orbits $M_0$ as the zero level-set of a function:
\[ g_0(z) := z_1^2 + z_2^2 +(z_3+\rho z_1^2)^2 - a^2 = 0\,. \]
Second, we define the corresponding flow \textit{on} this manifold, i.e., we need a vector field $T_0(z)$ that is tangent to $M_0$, $Dg_0(z) T_0(z)=0$. We choose:
\[ T_0(z) := \begin{bmatrix} -z_2 \\ z_1 \\ 2\rho z_1 z_2 \end{bmatrix}. \]
Finally, we define the dynamics \textit{off} the manifold. We add a component that vanishes on $M_0$ (via the $g_0(z)$ term) and directs the flow towards it. This normal component is given by $N_0(z) g_0(z)$, where:
$$
N_0(z) := -\frac{a}{2} \left(\frac{a + \sqrt{a^2 - (z_3 + \rho z_1^2)^2}}{z_3 + \rho z_1^2}\right) \begin{bmatrix} 0 \\ 0 \\ 1 \end{bmatrix}.
$$
Combining these pieces gives the full unperturbed vector field:
$$F_0(z) = T_0(z) + N_0(z) g_0(z)=
\begin{bmatrix}
    -z_2 \\
    z_1 \\
    2\rho z_1 z_2 - \frac{a}{2} \left(\frac{a + \sqrt{a^2 - (z_3 + \rho z_1^2)^2}}{z_3 + \rho z_1^2}\right) \left( z_1^2 + z_2^2 +(z_3+\rho z_1^2)^2 - a^2 \right)
\end{bmatrix}\,.$$
The perturbation term is chosen here as 
$$F_1(z) = 
\begin{bmatrix}
    0 \\ \frac{1}{a^2} \left( z_3 + \rho z_1^2 \right)^2 \left( \frac{1}{16}a^2 z_2 - z_1^2 z_2 \right)\\ 0
\end{bmatrix}
\,.$$

\begin{remark}[Domain of validity and smoothness]
The square root in the definition of $N_0(z)$ restricts the domain where the vector field is well-defined to the region $|z_3 + \rho z_1^2| \le a$. Furthermore, the vector field $F_0(z)$ is not smooth on the boundary of this domain, where $|z_3 + \rho z_1^2| = a$. As we will see, these boundary surfaces correspond to the poles of the manifold of periodic orbits.
\end{remark}

\begin{remark}
   To simplify the presentation, we set the initial time to be $t_0 = 0$ and the initial phase to be $\phi_0 = 0$. Additionally, we will use the shorthand notation $s_u := \sin u, c_u := \cos u$. 
\end{remark}

\noindent
The layer problem $z'=F_0(z)$, restricted to the manifold $M_0 = \{z \in \R^3 \mid g_0(z)=0\}$ is governed by the tangential vector field, 
$$z' = T_0(z)\,.$$ 
This flow consists of a 1-parameter family of periodic orbits. This manifold can be parametrised by
\begin{equation}
    M_0 = \left\{ z=\Gamma_0(x,\phi) = \begin{bmatrix} a s_x c_\phi \\ a s_x s_\phi \\ a c_x - \rho a^2 s^2_x \, c^2_{\phi}  \end{bmatrix}  : (x,\phi) \in (0,\pi) \times \mathbb R/2\pi \mathbb Z \right\},
\end{equation}
where $(x,\phi)$ is the parametrisation chart. For a fixed $x\in (0,\pi)$, the curve $\Gamma_0(x,\phi)$ is a periodic orbit with frequency $\omega(x)=1$, and thus period $\tau(x)=2\pi$.
The manifold is shown in Fig.~\ref{fig:ellipsoid}.

\begin{figure}
    \centering
    \includegraphics[width=9cm]{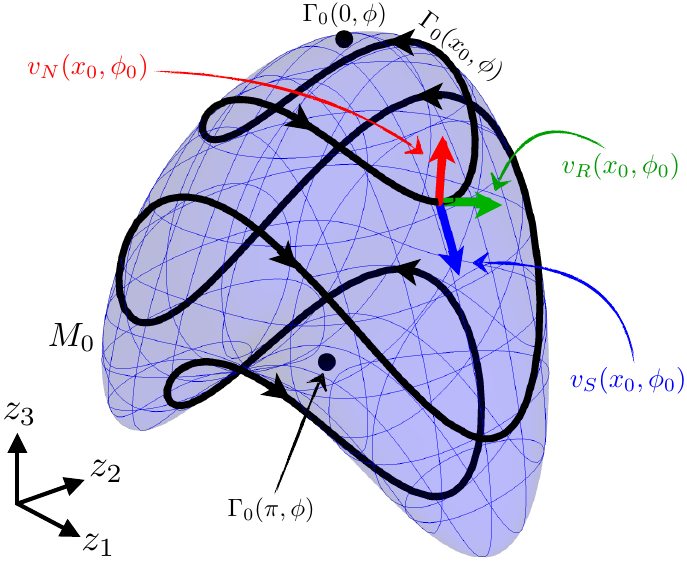}
    \caption{The bent ellipsoid manifold $M_0$ of periodic orbits for $a=1$ and $\rho = \tfrac{5}{3}$. Representative periodic orbits (black curves) are shown for three different values of $x$. The rotational, slow, and normal vectors $v_R(x_0,\phi_0), v_S(x_0,\phi_0)$ and $v_N(x_0,\phi_0)$ are shown in green, blue, and red, respectively at a fixed phase $\phi_0$. The vectors $v_R(x_0,\phi_0)$ and $v_S(x_0,\phi_0)$ span the tangent plane to the manifold $M_0$ at the point $\Gamma(x_0,\phi_0)$. The Jacobian $D\Gamma_0$ is rank deficient at the poles, i.e., at $x=0$ and $x=\pi$ (black markers), where the periodic orbits become fixed points.}
    \label{fig:ellipsoid}
\end{figure}

\begin{remark}[Singularity at the poles]
The domain for the slow variable is restricted to $x \in (0,\pi)$ to exclude the poles at $x=0$ and $x=\pi$. At these poles, the periodic orbits collapse to fixed points, causing the Jacobian of the parametrisation $D\Gamma_0$ to lose rank (i.e., the parametrisation is singular). Our analysis is therefore confined to the non-degenerate part of the manifold.
\end{remark}

\subsection{Bundle splitting and Floquet normal form}


\subsubsection{Basis of the tangent and normal bundles}
The basis vectors for the tangent and normal bundles are computed from the parametrisation $\Gamma_0(x,\phi)$.
\begin{itemize}
    \item \textbf{Rotational/Phase-drift vector $v_R$:} This vector is tangent to the periodic orbits.
    \[ v_R(x,\phi) = D_\phi \Gamma_0(x,\phi) = \begin{bmatrix} -a s_x \, s_\phi \\  a s_x \, c_\phi \\ \rho a^2 s^2_x \,  s_{2 \phi} \end{bmatrix}. \]
    A direct calculation confirms that $v_R(x,\phi) = T_0(\Gamma_0(x,\phi))$, consistent with the construction.

    \item \textbf{Slow/Phase-preserving vector $v_S$:} Since $\omega(x)=1$ is constant, the phase correction term is zero.
    \[ v_S(x,\phi) = D_x\Gamma_0(x,\phi) = \begin{bmatrix} a c_x \, c_\phi \\ a c_x \, s_\phi \\ -a s_x- \rho a^2 s_{2x}\, c^2_\phi \end{bmatrix}. \]

    \begin{remark}[On analytical tractability]
The choice of a constant frequency, $\omega(x)=1$, is a deliberate simplification crucial for keeping this example analytically solvable. A non-constant $\omega(x)$ would introduce a non-trivial phase correction term in the definition of $v_S$. This would create a cascade of complexity, making the analytical inversion of the matrix $P(x,t)$ and the evaluation of the subsequent integrals practically intractable. In Section~\ref{sec:mlt}, we will relax this condition and analyse a system with a non-trivial phase drift using numerical methods.
\end{remark}

    \item \textbf{Normal vector $v_N$:} This vector is defined by evaluating the normal component of the vector field on the manifold.
    \[ v_N(x,\phi) = N_0(\Gamma_0(x,\phi)) = \begin{bmatrix} 0 \\ 0 \\ \tfrac{1}{2} a \mu(x) \sec x \end{bmatrix}, \]
    where $\mu(x) = -(1+\sin x)$.
\end{itemize}

\subsubsection{Floquet normal form}
The transition matrix has the Floquet form $\Phi(x,t) = P(x,t) e^{B(x)t}$, where we use $t=\phi$ since $\omega=1$. The periodic matrix $P(x,t)$ is formed by the basis vectors:
\[ P(x,t) = \begin{bmatrix} v_R(x,t) & v_S(x,t) & v_N(x,t) \end{bmatrix}. \]
By construction of the vector field $F_0(z)$, the constant matrix $B(x)$ is diagonal in this basis:
\[ B(x) = P^{-1}(x,t) \left( DF_0 (\Gamma_0(x,t)) P(x,t) - D_t P(x,t) \right) = \begin{bmatrix} 0 & 0 & 0 \\ 0 & 0 & 0 \\ 0 & 0 & \mu(x)  \end{bmatrix}. \]
The two zero eigenvalues correspond to the neutral tangential directions, and $\mu(x) = -(1+\sin x)$ is the non-trivial Floquet exponent. Since $\mu(x) < 0$ for $x \in (0,\pi)$, the manifold $M_0$ is normally hyperbolic and attracting.

\subsection{Projection maps}
By construction, the columns of the matrix $P(x,t,0)$ form a basis for the ambient space $\mathbb R^3$.  
We can then compute a basis for the dual space (of the tangent space) as the rows of the inverse matrix, $P^{-1}(x,t,0)$. 
Using these bases, we construct the oblique projection maps as follows. \medskip

\subsubsection{Oblique projection onto the rotational bundle $\boldsymbol{TR_0}$}
The projection map onto the rotational bundle is given by the outer product of $v_R(x,t)$ and $w_R(x,t)$, where $w_R(x,t)$ is the first row of $P^{-1}(x,t,0)$:
\[ \Pi_R(x,t) = v_R(x,t) w_R(x,t) =
\begin{bmatrix}  
	\st^2  &  -\st \ct  &  0  \\ 
	-\st \ct  &  \ct^2  & 0  \\ 
	-2a \rho \ct \st^2 \sx  &  2a \rho \ct^2 \st \sx & 0
\end{bmatrix}.
\]
%

\subsubsection{Oblique projection onto the slow bundle $\boldsymbol{TS_0}$}
The projection map onto the phase-preserving slow bundle is the outer product of $v_S(x,t)$ and $w_S(x,t)$, where $w_S(x,t)$ is the second row of $P^{-1}(x,t,0)$:
\[ \Pi_S(x,t) = v_S(x,t) w_S(x,t) =
\begin{bmatrix}
	\ct^2  &  \st \ct  &  0 \\ 
	\st \ct  & \st^2  & 0  \\
	-\ct (1+2a \rho \cx \ct^2) \tx  &  -(1+2a \rho \cx \ct^2) \st \tx & 0 
\end{bmatrix}.
\]
%

\subsubsection{Oblique projection onto the normal bundle $\boldsymbol{\mathcal N}$}
The projection map onto the normal bundle is the outer product of $v_N(x,t)$ and $w_N(x,t)$, where $w_N(x,t)$ is the third row of $P^{-1}(x,t,0)$:
\[ \Pi_N(x,t) = v_N(x,t) w_N(x,t) = 
\begin{bmatrix}  
	0  &  0  &  0 \\
	0  &  0  &  0 \\
	(2a \rho \sx + \tx) \ct & \tx \st & 1
\end{bmatrix}.
\]
%

\subsection{Homological equation \& solution procedure}
The homological equation for the first-order correction to the manifold is
\[ \mathcal{L}_0 [\Gamma_1(x,\phi)] +D\Gamma_0(x,\phi) \, r_1 = F_1(\Gamma_0(x,\phi)) \]
where $\mathcal{L}_0$ is the linear operator $\mathcal{L}_0[\cdot] = \omega(x) D_{\phi} [\cdot] - DF_0(\Gamma_0) [\cdot ]$. The corresponding adjoint operator is 
\[ \mathcal{L}_0^*[\cdot] = - \omega(x) D_{\phi} [\cdot] -DF_0(\Gamma_0)^T \, [\cdot] \]
and the adjoint variational equation is $\mathcal{L}_0^*[\psi] = 0$.  

\subsubsection{Null space of the adjoint operator}
Given the transition matrix, $\Phi(x,t,0)$, of the primal variational equation, $D_t \Phi = DF_0(\Gamma_0) \, \Phi$, the transition matrix, $\Psi(x,t,0)$, of the adjoint variational equation is 
\[ \Psi(x,t,0) = (\Phi(x,t,0)^{-1})^T = \left( P(x,t,0)^{-1} \right)^T e^{-B(x)^T t}  P(x,0,0)^T.  \]
From this, we can immediately write down the associated periodic matrix
\begin{align*} 
\begin{bmatrix} \psi_1(x,t) & \psi_2(x,t) & \psi_3(x,t) \end{bmatrix} = 
\left(P(x,t,0)^{-1} \right)^T &= \frac{1}{a} \begin{bmatrix} 
  -\st/\sx  & \ct / \cx  & 2 \sx(1+2 a \rho \cx) \ct/\mu(x) \\ 
  \ct/\sx & \st/\cx & 2\sx \st / \mu(x) \\ 
  0 & 0 & 2\cx / \mu(x)
\end{bmatrix}.
\end{align*}
Thus, a periodic basis for the null space of the adjoint operator is $\left\{ \psi_1(x,t),\psi_2(x,t) \right\}$. 

We now follow the 3-step process of Section~\ref{subsec:solnprocedure} to construct a solution of the homological equation that achieves a maximally reduced flow. 
\medskip

\subsubsection{Step 1: Averaged dynamics}
By applying the Fredholm alternative to our homological equation, we have the constraint
\[ \begin{bmatrix} \left( \frac{1}{2\pi} \int_0^{2\pi} D\Gamma_0(x,\phi)^T  \psi_1(x,\phi) \, d\phi \right)^T \\ \left( \frac{1}{2\pi} \int_0^{2\pi} D\Gamma_0(x,\phi)^T  \psi_2(x,\phi) \, d\phi \right)^T \end{bmatrix}  r_1(x) = \begin{bmatrix} \frac{1}{2\pi} \int_0^{2\pi}  F_1(\Gamma_0(x,\phi))^T  \psi_1(x,\phi) \, d\phi \\ \frac{1}{2\pi} \int_0^{2\pi}  F_1(\Gamma_0(x,\phi))^T  \psi_2(x,\phi) \, d\phi \end{bmatrix}. \]
Evaluating the integrals, the equation for the correction to the vector field simplifies to
\[ \begin{bmatrix}  0 & 1 \\ 1 & 0 \end{bmatrix} r_1(x) = \begin{bmatrix} 0 \\ \tfrac{1}{64} a^2 (s_{4x}-s_{2x}) \end{bmatrix}. \]
Hence, the first-order correction to the maximally reduced vector field is
\[ r_1(x) = \begin{bmatrix} \tfrac{1}{64} a^2 (s_{4x}-s_{2x}) \\ 0  \end{bmatrix}. \]
%

\subsubsection{Step 2: Normal correction}
By projecting the homological equation onto the normal bundle, we obtain the first-order, linear, inhomogeneous ODE
\[ \Pi_N \mathcal{L}_0 [\Gamma_1^N] = \Pi_N (x,\phi) F_1(\Gamma_0(x,\phi)) \]
where $\Gamma_1^N(x,\phi) = \Pi_N (x,\phi) \Gamma_1(x,\phi)$ is the projection of $\Gamma_1(x,\phi)$ onto the normal bundle. 
From the variation of parameters formula given in Section~\ref{subsec:solnprocedure}, the normal correction is given by 
\[ \Gamma_1^N(x,\phi) = \Phi_N(x,\phi,0) \Gamma_1^N(x,0) + \frac{1}{\omega(x)} \int_{0}^{\phi} \Phi_N(x,\phi,\sigma) \Pi_N(x,\sigma) F_1(\Gamma_0(x,\sigma))\,d\sigma,  \]
where $\Phi_N(x,\phi,0) = \Pi_N(x,t,0) \Phi(x,t,0) \Pi_N(x,0,0)$ is the projection of the transition matrix onto the normal bundle and is given by
\[ \Phi_N(x,\phi,0) = \begin{bmatrix} 
  0 & 0 & 0 \\
  0 & 0 & 0 \\ 
  e^{\mu(x) \phi} (2a \rho \sx + \tx) & 0 & e^{\mu(x) \phi}
\end{bmatrix}. \]
Enforcing $2\pi$-periodicity in $\phi$ uniquely determines the initial condition as 
\[ \Gamma_1^N(x,0) = \left( \mathbb I_3 -M_N(x)  \right)^{-1} \int_0^{2\pi} \Phi_N(x,2\pi,\sigma) \Pi_N(x,\sigma) F_1(\Gamma_0(x,\sigma))\, d\sigma  \]
where 
\[ M_N(x) = \begin{bmatrix}  
0 & 0 & 0 \\ 
0 & 0 & 0 \\
\tfrac{1}{4} a^2 e^{2\pi \mu(x)} \mu(x)^2 (2a\rho \sx + \tx) /\cx^2 & 0 & \tfrac{1}{4} a^2 e^{2\pi \mu(x)} \mu(x)^2/\cx^2
\end{bmatrix} \]
is the projection of the monodromy onto the normal bundle and 
\[ \Pi_N(x,\phi) F_1(\Gamma_0(x,\phi)) =  \begin{bmatrix} 
0 \\ 0 \\ \frac{1}{16} a^3 \cx \sx^2 \sphi^2 \left(1 - 16 \sx^2 \cp^2 \right) 
\end{bmatrix}. \]
Substituting into the solution formula and simplifying, the normal correction is  
\[ \Gamma_1^N(x,\phi) = 
\frac{a^3 \sx^2}{8\mu(4+\mu^2)(16+\mu^2)}
\begin{bmatrix} 0 \\ 0 \\ h_0 + h_1\, \mu + h_2\, \mu^2 + h_3\, \mu^3 + h_4 \, \mu^4 \end{bmatrix}, \]
where the coefficient functions are
\begin{equation*} 
  \begin{split}
      h_0 = -32+64c_{2x}, \quad 
      h_1 =& 16 (s_{2\phi}-2\sx^2 s_{4\phi}), \quad 
      h_2 = -10+20 c_{2x}-8 c_{2\phi} +4 c_{4\phi} -4 c_{2x}c_{4\phi}, \\ 
      h_3 &= s_{2\phi}-8\sx^2 s_{4\phi}, \quad {\rm and } \quad 
      h_4 = \sphi^2 - 4\sx^2 s_{2\phi}^2, 
  \end{split}
\end{equation*}
and we have suppressed the $x$-dependence of $\mu$.

\subsubsection{Step 3: Tangential correction}

Having determined the maximally reduced vector field, the homological equation is now fully determined and is given by the first-order, linear, inhomogeneous ODE
\[ \mathcal{L}_0[ \Gamma_1] = F_1(\Gamma_0) - D \Gamma_0 r_1.  \]

To find the unique periodic solution, our approach combines the final two steps of the procedure from Section~\ref{subsec:solnprocedure}, i.e., we will solve for the full correction $\Gamma_1$ directly using the variation of parameters formula, applying both the periodicity requirement and the orthogonality condition from Equation~\eqref{eq:gauge_condition} simultaneously to determine the initial condition $\Gamma_1(x,0)$.

\noindent
Solving by variation of parameters, we have
\[ \Gamma_1(x,\phi) = \Phi(x,\phi,0) \Gamma_1(x,0) + J(x,\phi), \]
where 
\begin{equation*}
\begin{split} 
J(x,\phi) &:= \frac{1}{\omega(x)} \int_0^{\phi} \Phi(x,\phi,\sigma) \left( F_1(\Gamma_0(x,\sigma)) - D \Gamma_0(x,\sigma) r_1(x) \right) \, d\sigma \\ 
&= \frac{a^3 \sx}{8} \begin{bmatrix} \frac{1}{4}\cx^2 \sphi (2-3c_{2x}-2\sx^2 c_{2\phi}) \\ -3 \sx^2 \cx^2 \cp \sphi^2 \\ 
J_3
\end{bmatrix} \, ,
\end{split}
\end{equation*}
and 
\[ J_3 = \tfrac{1}{2} s_{2\phi}+ \rho a (\sx^2s_{4\phi}+s_{2\phi}(3c_{2x}-2)) 
- \tfrac{2\sphi(2\cp+\sphi \mu)}{4+\mu^2} 
- \tfrac{\sx^2 \mu (\mu s_{4\phi}-8 s_{2\phi}^2)}{16+\mu^2}
+ \tfrac{4(e^{\mu \phi}-1) (32 c_{2x}-16+(8 c_{2x}-7)\mu^2)}{\mu(4+\mu^2)(16+\mu^2)}. \]
Enforcing $2\pi$-periodicity in $\phi$ leads to the underdetermined system
\[ \begin{bmatrix} 0 & 0 & 0 \\ 0 & 0 & 0 \\ -M_{31}(x) & 0 & 1-M_{33}(x) \end{bmatrix} \Gamma_1(x,0) = J(x,2\pi) = \begin{bmatrix} 0 \\ 0 \\ \frac{a^3 \cx \sx^2(e^{2\pi \mu}-1)(8 c_{2x} (4+\mu^2)-16-7\mu^2)}{8 \mu (4+\mu^2)(16+\mu^2)} \end{bmatrix} ,\]
where $M_{ij}(x)$ denotes the $(i,j)$-element of the monodromy matrix $M(x) = \Phi(x,\omega(x)^{-1} 2\pi,0)$.
In order to specify a unique solution, we append the condition that $\Gamma_1(x,\phi)$ must be orthogonal to the adjoint null space. 
Thus, the linear system that uniquely determines the initial conditions is
\[ 
\begin{bmatrix} 
    -M_{31}(x) & 0 & 1-M_{33}(x) \\ 
    \langle \Phi_1, \psi_1 \rangle & \langle \Phi_2, \psi_1 \rangle & \langle \Phi_3, \psi_1 \rangle \\
    \langle \Phi_1, \psi_2 \rangle & \langle \Phi_2, \psi_2 \rangle & \langle \Phi_3, \psi_2 \rangle 
\end{bmatrix}
\Gamma_1(x,0) = \begin{bmatrix} J_3(x,2\pi) \\ -\langle J(x,\phi), \psi_1 \rangle \\ -\langle J(x,\phi), \psi_2 \rangle \end{bmatrix} 
= 
\begin{bmatrix} \frac{a^3 \cx \sx^2(e^{2\pi \mu}-1)(8 c_{2x} (4+\mu^2)-16-7\mu^2)}{8 \mu (4+\mu^2)(16+\mu^2)} \\ \frac{1}{64} a^2 \cx^2(4-5 c_{2x}) \\ 0 \end{bmatrix},
\]
where $\Phi_j$ denotes the $j^{\rm th}$ column of $\Phi(x,\phi,0)$. Evaluating the integrals, the system simplifies to
\[ 
\begin{bmatrix} 
    -M_{31}(x) & 0 & 1-M_{33}(x) \\ 
    0 & \tfrac{1}{a \sx} & 0 \\
    \tfrac{1}{a \cx} & 0 & 0
\end{bmatrix}
\Gamma_1(x,0) 
= 
\begin{bmatrix} \frac{a^3 \cx \sx^2(e^{2\pi \mu}-1)(8 c_{2x} (4+\mu^2)-16-7\mu^2)}{8 \mu (4+\mu^2)(16+\mu^2)} \\ \frac{1}{64} a^2 \cx^2(4-5 c_{2x}) \\ 0 \end{bmatrix},
\]
and hence the initial condition is 
\[ \Gamma_1(x,0) = a^3 \sx \cx \begin{bmatrix} 0 \\ \tfrac{1}{64} \cx (4-5 c_{2x}) \\ \frac{\sx (16+7\mu^2-8 c_{2x} (4+\mu^2))}{8\mu (4+\mu^2)(16+\mu^2)} \end{bmatrix}. \]
Thus, the unique periodic correction, $\Gamma_1(x,\phi)$, which is orthogonal to the adjoint null space, is 
\[ \Gamma_1(x,\phi) = \frac{a^3 \sx}{64} \begin{bmatrix} 
-\cx^2 (c_{2x}+4 \sx^2 c_{2\phi}) \sphi \\ 
\cx^2 \cp (4-5 c_{2x}-24 \sx^2 \sphi^2) \\ 
s_{2x} \left( \ell_1 - \frac{\ell_2}{4+\mu^2} - \frac{\ell_3}{16+\mu^2} + \frac{\ell_4}{\mu(4+\mu^2)(16+\mu^2)}  \right)
 \end{bmatrix}, \]
where the coefficient functions are 
\begin{equation*}
  \begin{split}
    \ell_1 &= \tfrac{1}{2} s_{2\phi} \left( 1+\rho a \cx (c_{2x}+4 \sx^2 c_{2\phi}) \right), \quad
    \ell_2 = 2 \sphi(2\cp+\mu \sphi), \\
    \ell_3 &= \mu \sx^2(4 c_{4\phi}-4+\mu s_{4\phi}), \quad {\rm and } \quad 
    \ell_4 = 4(16+7\mu^2-8c_{2x}(4+\mu^2)).
  \end{split}
\end{equation*}
The perturbed manifold $\Gamma_0(x,\phi) + \eps\, \Gamma_1(x,\phi)$, and the averaged flow along it, is shown in Fig.~\ref{fig:perturbedcroissant}. 

\begin{figure}[th]
    \centering
    \includegraphics[width=5in]{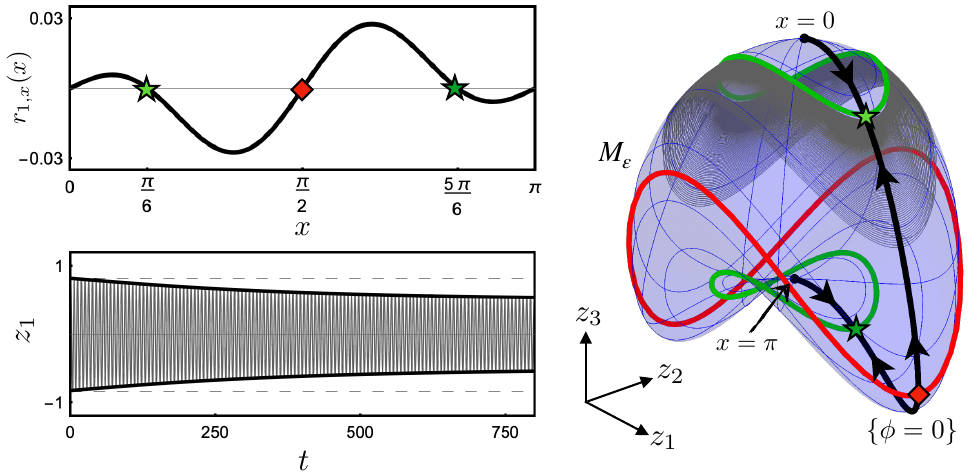}
    \put(-362,170){(a)}
    \put(-148,170){(b)}
    \put(-360,78){(c)}
    \caption{Averaged flow on the perturbed manifold for $a=1, \rho=\tfrac{5}{3}$, and $\eps = 0.05$. (a) The averaged slow vector field has stable fixed points at $x=\tfrac{\pi}{6}$ and $x=\tfrac{5\pi}{6}$ (green stars) and an unstable fixed point at $x=\tfrac{\pi}{2}$ (red diamond). (The equilibria at the poles are also unstable.) (b) Slow manifold $M_{\eps}$ computed using the first two terms, $\Gamma_0$ and $\Gamma_1$, of the embedding. For the chosen parameter values, we find that $\lVert \Gamma_1(x,\phi) \rVert \ll \lVert \Gamma_0(x,\phi) \rVert$ for all $(x,\phi) \in (0,\pi) \times [0,2\pi]$. Thus, the perturbed manifold $M_{\eps}$ appears very similar to the unperturbed manifold $M_0$. 
    The averaged flow (black arrows) is shown along the section $\{ \phi = 0 \}$ (black curve). The closed red and green curves are the periodic solutions corresponding to the fixed points in (a).
    A solution of the full system, $z^\prime(t) = F_0(z)+\eps F_1(z)$, with initial condition given by $z_1(0) = \Gamma_0(\tfrac{\pi}{3})+\begin{bmatrix} -0.1 & -0.3 & 0.5\end{bmatrix}^T$ is also shown (light gray curve). The solution rapidly contracts to $M_{\eps}$ and then slowly approaches the green stable limit cycle at $x=\tfrac{\pi}{6}$. (c) Time series of the $z_1$-component of the gray solution from (b). The horizontal dashed lines correspond to the initial amplitude and have been added to better illustrate the convergence of the solution to the attractor approximately given by $\Gamma_0(\tfrac{\pi}{6},\phi)+\eps \Gamma_1(\tfrac{\pi}{6},\phi)$.}
    \label{fig:perturbedcroissant}
\end{figure}

\noindent
To complete the calculation, the tangential component of the correction can now be computed as 
\[ \Gamma_1^M = \Pi_M \Gamma_1 = \Gamma_1 - \Gamma_1^N. \]
We find that it is given by
\[ \Gamma_1^M = \frac{a^3 \sx \cx}{64} 
\begin{bmatrix} 
  -\cx \sphi (c_{2x}+4 \sx^2 c_{2\phi}) \\ 
  \cx \cp (4-5 c_{2x}-24 \sx^2 \sphi^2) \\ 
  \tfrac{1}{2} \sx s_{2\phi} \left( 2-8 \sx^2 c_{2\phi} +\rho a (\cx + c_{3x}+8 \cx \sx^2 c_{2\phi}) \right)
\end{bmatrix}. \]
%

\section{Numerical example: Morris-Lecar-Terman model} \label{sec:mlt}
We now apply our methods to an extension of the Morris-Lecar model for neural excitability, known as the Morris-Lecar-Terman model \cite{Terman1991}, in which the constant applied current is replaced with a linear feedback control. The dimensionless Morris-Lecar-Terman model is given by 
\begin{equation} \label{eq:MLT}
  \begin{split}
	\dot v &= x - g_L(v-V_L)-g_K w(v-V_K)-g_{Ca} m_{\infty}(v) (v-V_{Ca}) \\ 
	\dot w &= \frac{w_{\infty}(v)-w}{\tau_w(v)} \\ 
	\dot x &= \eps \left( k-v \right)
  \end{split}
\end{equation}
where $v$ is the voltage, $w$ is the recovery variable, and $x$ is the externally applied current. The steady-state activation functions are given by
\[ m_{\infty}(v) = \frac{1}{2} \left( 1+\tanh \left( \frac{v-c_1}{c_2} \right) \right) \quad {\rm and } \quad w_{\infty}(v) = \frac{1}{2} \left( 1+\tanh \left( \frac{v-c_3}{c_4} \right) \right), \]
and the voltage-dependent time-scale function is 
\[ \tau_w(v) = \tau_0 \operatorname{sech} \left( \frac{v-c_3}{2c_4} \right). \]
Standard parameter values are listed in Table~\ref{tab:MLT}. The parameter $k$ is usually taken to be the principal bifurcation parameter. 
The Morris-Lecar-Terman model can exhibit a wide variety of oscillatory dynamics under variations in $k$, include spiking, fold/homoclinic bursting, and circle/fold-cycle bursting. 
Here, we fix $k=-0.12$ so that the slow drift along the manifold of periodics stays away from any bifurcations of the periodics.

\begin{table}[ht!]
   \centering
   \begin{tabular}{clclclcl} 
	\hline
	Param. & Value & Param. & Value & Param. & Value & Param. & Value\\ 
	\hline
	$g_L$ & $0.5$ & $g_K$ & $2$ & $g_{Ca}$ & $1.25$ & $V_L$ & $-0.5$  \\ 
	$V_K$ & $-0.7$ & $V_{Ca}$ & $1$ & $c_1$ & $-0.01$ & $c_2$ & $0.15$ \\ 
	$c_3$ & $0.1$  & $c_4$ & $0.16$ & $\tau_0$ & $3$ & $\eps$ & $0.01$ \\
	\hline
   \end{tabular}
   \caption{Standard parameter values for the Morris-Lecar-Terman model.}
   \label{tab:MLT}
\end{table}

\noindent
The Morris-Lecar-Terman model is a slow/fast system in standard form with
\[ z^{\prime} = F_0(z) + \eps F_1(z) = \begin{bmatrix} x - g_L(v-V_L)-g_K w(v-V_K)-g_{Ca} m_{\infty}(v) (v-V_{Ca}) \\ \frac{w_{\infty}(v)-w}{\tau_w(v)}  \\ 0 \end{bmatrix} + \eps \begin{bmatrix} 0 \\ 0 \\ k-v  \end{bmatrix}, \]
where $z = (v,w,x)^T$.

\subsection{Numerical formulation}	\label{sec:mlt_numerics}
To compute periodic solutions $\gamma(x,t)$ of the layer problem, together with solutions $\Phi(x,t,0)$ of the associated variational equation, we solved the following boundary value problem
\begin{equation} 	\label{eq:mlt_bvp}
  \begin{split}
	\frac{dz}{d\widetilde t} &= T \, F_0(z) \\ 
	\frac{\partial \Phi}{\partial \widetilde t} &= T\, DF_0(z)\, \Phi
  \end{split}
\end{equation}
where $T$ is the $x$-dependent integration time (i.e., the period), $\widetilde t = t/T \in [0,1]$ is the rescaled time, and the boundary conditions are
\begin{equation}	\label{eq:mlt_bcs}
  \begin{split}
	z(0) - z(1) &= 0 \\ 
	\Phi(x,0,0) &= \mathbb I_3.
  \end{split}
\end{equation}
The boundary conditions in \eqref{eq:mlt_bcs}(a) enforce periodicity of the solutions of the layer problem, and the boundary conditions in \eqref{eq:mlt_bcs}(b) guarantee that $\Phi(x,t,0)$ is the transition matrix. 

\begin{remark} \label{remark:PhiRow3}
Since the component of $F_0$ corresponding to the slow direction is trivial, the third row of $\Phi(x,t,0)$ is particularly simple, with
\[ \Phi_{31}(x,t,0) \equiv 0,  \quad \Phi_{32}(x,t,0) \equiv 0, \quad {\rm and} \quad \Phi_{33}(x,t,0) \equiv 1. \]
\end{remark}

To generate the manifold of periodics of the layer problem, we first computed a solution of \{\eqref{eq:mlt_bvp},\,\eqref{eq:mlt_bcs}\} for a fixed value of $x$ using an initial value solver. Then, we used the numerical continuation software package \textsc{Auto} \cite{auto07p} to numerically continue our starting solution in the parameters $x$ and $T$. 
To guarantee uniqueness of our numerically computed solutions, we impose the integral phase condition
\[ \int_0^1 \frac{dz_{\rm old}}{d\widetilde{t}} \cdot \left(z-z_{\rm old} \right)\,d\widetilde{t} = 0, \]
where $z_{\rm old}$ refers to the solution at the previous continuation step.
In this way, we were able to compute the manifolds of periodics, along with their families of transition matrices. The manifolds of periodics and representative elements of $\Phi(x,t,0)$ are shown in Fig.~\ref{fig:mltlayer}. 

\begin{figure}[ht!]
   \centering
   \includegraphics[width=5in]{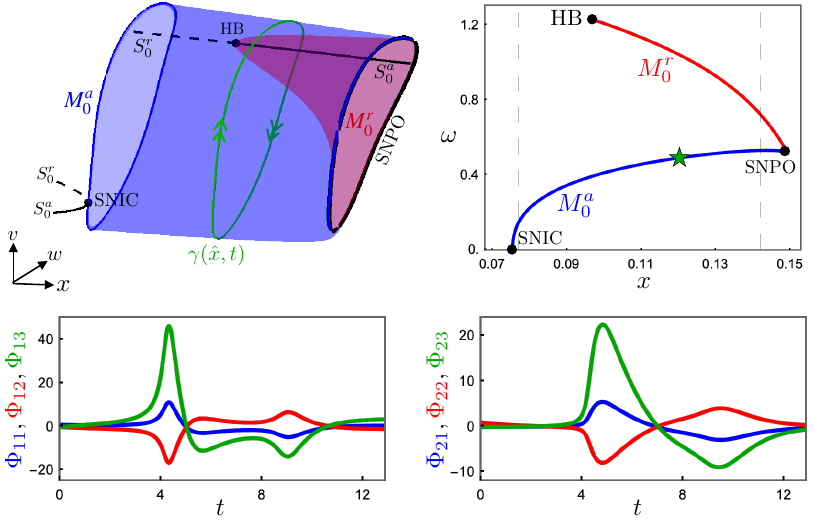}
   \put(-360,222){(a)}
   \put(-170,222){(b)}
   \put(-360,90){(c)}
   \put(-180,90){(d)}
   \caption{Properties of the layer problem of the Morris-Lecar-Terman model. 
   (a) The critical manifold ($S_0^a \cup S_0^r$; black solid and dashed curves) undergoes a subcritical Hopf bifurcation (HB) at $x_{\rm HB} \approx 0.0973$ and a saddle-node on invariant circle (SNIC) bifurcation at $x \approx 0.07544$. The (red) repelling manifold, $M_0^r$, of periodics emanates from the Hopf bifurcation and connects to the (blue) attracting manifold, $M_0^a$, of periodics at a saddle-node of periodic orbits (SNPO) at $x \approx 0.1493$. A representative normally hyperbolic limit cycle $\gamma(\hat{x},t)$ is highlighted in green for $\hat x = 0.12$. 
   (b) Frequency-current curves for $M_0^a$ and $M_0^r$. The attracting manifold is delimited by the SNIC bifurcation (where the frequency goes to zero) and the SNPO bifurcation (where there is loss of smoothness in the frequency map). The green star marker indicates the representative limit cycle $\gamma(\hat{x},t)$.
   (c) First row of the transition matrix $\Phi(\hat x,t,0)$ of the variational equation along $\gamma(\hat{x},t)$. 
   (d) Second row of $\Phi(\hat x,t,0)$.}
   \label{fig:mltlayer}
\end{figure}
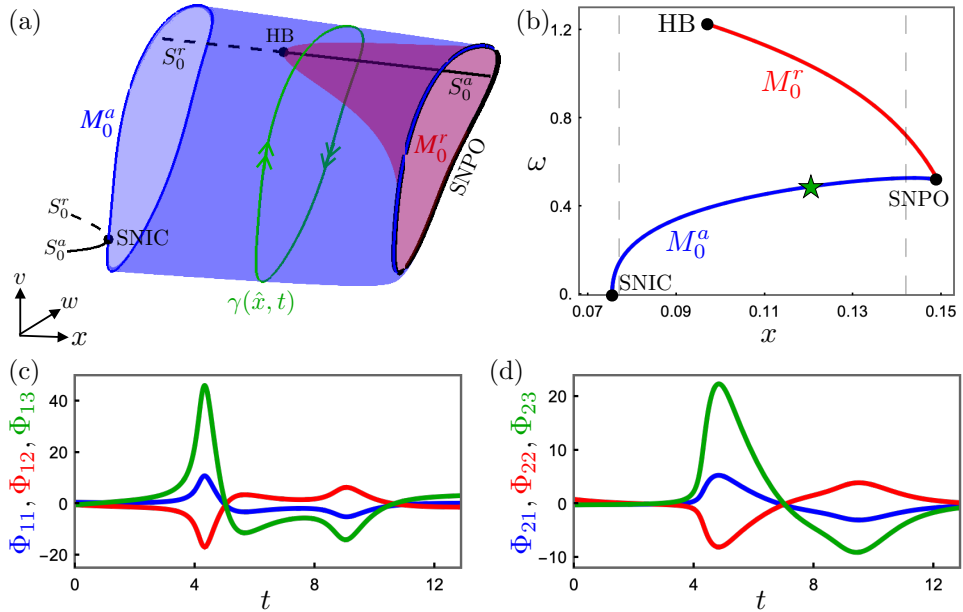

The layer problem, $z^\prime = F_0(z)$, of the Morris-Lecar-Terman model possesses two families of limit cycles, $M_0^a$ and $M_0^r$ (blue and red surfaces, respectively, in Fig.~\ref{fig:mltlayer}(a)), each parametrized by the slow variable $x$. 
The repelling manifold $M_0^r$ emanates from the subcritical Hopf bifurcation and terminates at a saddle-node of periodic orbits (SNPO) bifurcation, where it merges with the attracting manifold $M_0^a$. 
The boundaries of the attracting manifold are the SNPO bifurcation and the saddle-node on invariant circle (SNIC) bifurcation (where $M_0^a$ is homoclinic to the critical manifold $S_0$). 
Here, we restrict our attention to the $x$-interval $U := [0.077,0.142]$ so that $M_0^a$ is normally hyperbolic attracting for all $x \in U$. The manifold $M_0^r$ is normally hyperbolic repelling for all $x \in (x_{HB},0.142]$ where $x_{HB} \approx 0.0973$ is the $x$-coordinate of the Hopf bifurcation.  

\begin{remark}
The methods developed in this article break down at the SNPO and SNIC bifurcations. 
At the SNPO bifurcation, there is loss of normal hyperbolicity and the construction of the normal bundle requires modification. 
As such, our methods cannot resolve the dynamics of solutions in their passage past the SNPO region. 
That is, we are unable to study any torus canard dynamics. 
At the SNIC bifurcation, the frequency $\omega(x)$ goes to zero and its slope $\omega^\prime(x)$ becomes infinitely steep. Consequently, our algorithms for the constructions of the phase-preserving vectors, $v_S$, and the corrections, $\Gamma_j$, are no longer valid. We leave the adaptation of our methods for these non-hyperbolic scenarios to future work. 
\end{remark}

The frequency-current curve is shown in Fig.~\ref{fig:mltlayer}(b). The blue branch corresponds to the frequency map of the attracting manifold $M_0^a$ and the red branch corresponds to the frequency map of the repelling manifold $M_0^r$. Provided $x$ is chosen away from the SNIC and SNPO bifurcations, the frequency map of $M_0^a$ is smooth and bounded. Our chosen $x$-interval $U$ is indicated by the vertical dashed lines. 

We now proceed to describe the methods we used to compute the monodromy matrix, the Floquet normal form for the transition matrix, and the slow, rotational, and normal vectors.  

\subsection{Monodromy matrix and Floquet normal form}
By the set-up established in Section~\ref{sec:mlt_numerics}, we have that the monodromy matrix $\Phi(x,T,0)$ is simply the (right) $\widetilde t$-endpoint data of our $\Phi$ solutions. 

To compute the Floquet normal form of $\Phi$, we first use the following relation for the monodromy: 
\[ \Phi(x,T,0) = e^{T B(x)}, \] 
where $B(x)$ is the constant matrix in Floquet's Theorem. 
Thus, for each fixed $x$, we calculated $B(x)$ as
\[ B(x) = \frac{1}{T} \log \Phi(x,T,0). \]
Here, the matrix logarithm was implemented using MATLAB's built-in (principal) matrix logarithm function, \texttt{logm}. Then, again for each fixed $x$, we computed the periodic matrix $P(x,t,0)$ in the Floquet normal form as
\[ P(x,t,0) = \Phi(x,t,0) e^{-B(x) t}. \]
By Remark~\ref{remark:PhiRow3}, the last row of $P(x,t,0)$ is simple with entries given by 
\[ P_{31}(x,t,0) \equiv 0, \quad P_{32}(x,t,0) \equiv 0, \quad {\rm and } \quad P_{33}(x,t,0) \equiv 1 \quad \text{ for all } \,\,x \in U. \]
%

\subsection{Rotational, slow, and normal vectors}
We now use the periodics, transition matrices, and Floquet normal form to construct the rotational, slow, and normal vectors. For each fixed $x$:   
\begin{itemize}
\setlength{\itemsep}{0pt}
\item We computed the rotational (i.e., velocity) vector as $v_R(x,t) = F_0(\gamma(x,t))$. 
\item We computed the phase-preserving slow vectors as 
\[ v_S(x,t) = \Phi(x,t,0) v_S(x,0) - t \frac{\omega^\prime(x)}{\omega(x)} v_R(x,t), \]
where $v_S(x,0) = D_x \gamma(x,0)$. To calculate the derivatives $D_x\gamma(x,0)$ and $\omega^\prime(x)$, we used interpolation to produce numerically differentiable interpolants. (The data for $\gamma(x,0)$ is obtained from the solution family in Section~\ref{sec:mlt_numerics} as the left $\widetilde t$-endpoint data and the data for $\omega(x)$ is obtained as $\frac{2\pi}{T}$, where $T$ is the total integration time.)  
\item We computed the normal vectors as 
\[ v_N(x,t) = P(x,t,0) v_N(x,0), \]
where $P$ is the periodic matrix from the Floquet normal form, and $v_N(x,0)$ was chosen to be the eigenvector of $B(x)$ corresponding to the non-zero eigenvalue. 
\end{itemize}
%
%

\subsection{Oblique projection maps}
Having constructed the rotational, slow, and normal vectors, we now compute the oblique projection maps. To do this, we iterated the following steps for each fixed coordinate pair $(\hat x,\hat t) \in U \times [0,T]$. 
\begin{enumerate}
\item Construct the matrix 
\[ W(\hat x,\hat t) = \begin{bmatrix} v_R(\hat x, \hat t) & v_S(\hat x, \hat t) & v_N(\hat x, \hat t) \end{bmatrix}, \] 
the columns of which form a basis for the ambient space $\mathbb R^3$. 
\item Compute the inverse matrix, $W^{-1}(\hat x,\hat t)$, the rows of which form a basis for the dual space. 
\item Compute the projection maps $\Pi_R(\hat x,\hat t), \Pi_S(\hat x,\hat t)$, and $\Pi_N(\hat x,\hat t)$ as the outer products 
\begin{equation*} 
  \begin{split}
    \Pi_R(\hat x,\hat t) &= v_R(\hat x, \hat t) w_R(\hat x, \hat t) \\
    \Pi_S(\hat x,\hat t) &= v_S(\hat x, \hat t) w_S(\hat x, \hat t) \\
    \Pi_N(\hat x,\hat t) &= v_N(\hat x, \hat t) w_N(\hat x, \hat t) 
  \end{split}
\end{equation*}
where $w_R(\hat x, \hat t), w_S(\hat x, \hat t)$, and $w_N(\hat x, \hat t)$ respectively denote the first, second, and third rows of $W^{-1}(\hat x,\hat t)$.
\end{enumerate}

\subsection{Homological equation \& solution procedure}
The homological equation for the first-order correction to the manifold is
\[ \mathcal{L}_0 [\Gamma_1(x,\phi)] + D\Gamma_0(x,\phi) r_1(x) = F_1 (\Gamma_0(x,\phi)) \]
where $\mathcal{L}_0[\cdot] = \omega(x) D_{\phi} [\cdot] - DF_0(\Gamma_0) [\cdot]$. The corresponding adjoint operator is $\mathcal{L}_0^*[\cdot] = -\omega(x) D_{\phi}[\cdot] - DF_0(\Gamma_0)^T [\cdot]$. The transition matrix of the adjoint variational equation, $\mathcal{L}_0^*[\psi] = 0$, is given by 
$\Psi(x,t,0) = \left( \Phi(x,t,0)^{-1} \right)^T$. 
By choosing the following representation for the Floquet normal form, 
\[ \Phi(x,t,0) = W(x,t) e^{C(x) t} W^{-1}(x,0), \] 
where $W = \begin{bmatrix} v_R & v_S & v_N \end{bmatrix}$ and $C$ is a constant (in time) matrix, 
the periodic matrix of the adjoint variational equation is
\[ \begin{bmatrix} \psi_1(x,t) & \psi_2(x,t) & \psi_3(x,t) \end{bmatrix} = \left( W(x,t)^{-1} \right)^T. \]
%
In this formulation, the first two columns $\left\{ \psi_1(x,t), \psi_2(x,t) \right\}$ form a periodic basis for the null space of the adjoint linear operator. 

\begin{remark}
The periodic and constant matrices in a Floquet decomposition of the transition matrix are not uniquely defined, which is what allows us to choose our representation of $\Phi$ in terms of the basis vectors in $W$. The Floquet multipliers, however, are unique. 
\end{remark}

We now construct a solution of the homological equation that achieves a maximally reduced flow on $M_{\eps}$ by following the 3-step iterative procedure described in Section~\ref{subsec:solnprocedure}. 

\subsubsection{Step 1: Averaged dynamics}
By the Fredholm alternative, the homological equation is solvable provided 
\[ \begin{bmatrix} \left( \frac{1}{2\pi} \int_0^{2\pi} D\Gamma_0(x,\phi)^T  \psi_1(x,\phi) \, d\phi \right)^T \\ \left( \frac{1}{2\pi} \int_0^{2\pi} D\Gamma_0(x,\phi)^T  \psi_2(x,\phi) \, d\phi \right)^T \end{bmatrix}  r_1(x) = \begin{bmatrix} \frac{1}{2\pi} \int_0^{2\pi}  F_1(\Gamma_0(x,\phi))^T  \psi_1(x,\phi) \, d\phi \\ \frac{1}{2\pi} \int_0^{2\pi}  F_1(\Gamma_0(x,\phi))^T  \psi_2(x,\phi) \, d\phi \end{bmatrix}. \]
Here, we computed the Jacobian $D\Gamma_0(x,\phi)$ using the relation
\[ D\Gamma_0(x,\phi) = \begin{bmatrix} D_x \Gamma_0(x,\phi) & D_{\phi} \Gamma_0(x,\phi) \end{bmatrix} = \begin{bmatrix} v_S\left( x,\tfrac{\phi}{\omega(x)} \right) & \tfrac{1}{\omega(x)} v_R\left( x,\tfrac{\phi}{\omega(x)} \right) \end{bmatrix}  \]
and we evaluated the integrals using MATLAB's in-built trapezoidal integration routine. The components of the averaged vector field are shown in Fig.~\ref{fig:mlt_vectorfield}. 

\begin{figure}[ht!]
    \centering
    \includegraphics[width=5in]{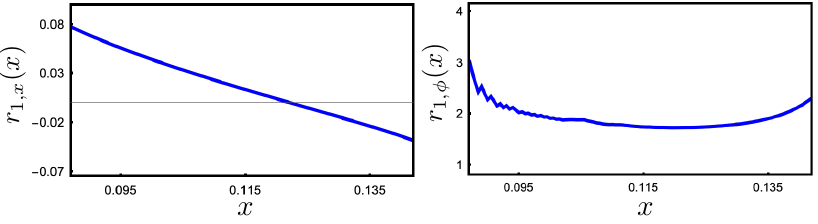}
    \put(-362,91){(a)}
    \put(-176,91){(b)}
    \caption{Averaged vector field $r_1(x)$ for $k=-0.12$. (a) The $x$-component and (b) $\phi$-component of $r_1$.}
    \label{fig:mlt_vectorfield}
\end{figure}

The relationship between our maximally reduced vector field computed using the Fredholm alternative and the classical averaging method is illustrated in Fig.~\ref{fig:mlt_averaged}. From the averaged vector field in Fig.~\ref{fig:mlt_averaged}(a), we find that there is a stable averaged equilibrium. 
Hence, the slow drift along a section $\{ \phi = {\rm constant} \}$ of the manifold will converge to the $x$-value corresponding to the averaged equilibrium, as shown in panels (b) and (c). 

\begin{figure}[ht!]
   \centering
   \includegraphics[width=3.5in]{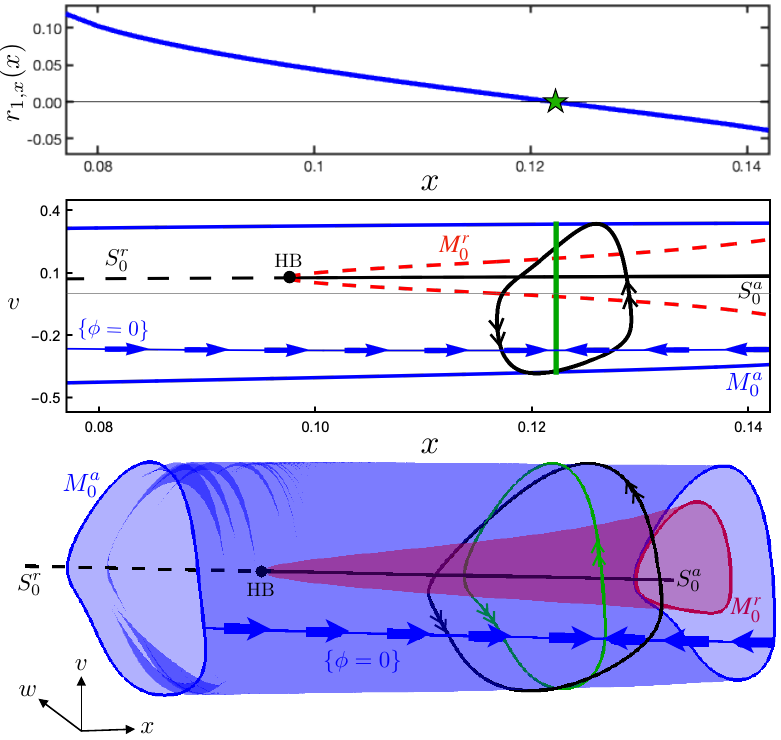}
   \put(-256,232){(a)}
   \put(-256,168){(b)}
   \put(-256,85){(c)}
   \caption{Slow averaged dynamics of the Morris-Lecar-Terman model for $k=-0.12$ on the restricted $x$-interval $U = [0.077,0.142]$, where $\left. M_0^a \right|_U$ is strictly normally hyperbolic attracting. (a) The slow vector field $r_{1,x}(x)$ possesses a stable equilibrium (green marker). (b) Projection of the dynamics into the $(x,v)$ phase plane. The critical manifold is unstable (dashed, black curve $S_0^r$) to the left of the subcritical Hopf bifurcation and is stable (solid, black curve $S_0^a$) to the right of it. The maximum and minimum $v$-coordinates of the limit cycles in the attracting and repelling manifolds of periodics, $M_0^a$ and $M_0^r$, are shown in blue and red, respectively. We also show the $\{ \phi=0 \}$-section of $M_0^a$ (thin blue curve) together with the (scaled) slow vector field (thick blue arrows) along it. The stable fixed point of $r_{1,x}(x)$ corresponds to a specific limit cycle in $M_0^a$, which we highlight in green. For comparison, we include the full system spiking attractor for $\eps = 0.01$ (black orbit), which is an $\mathcal{O}(\eps)$ perturbation of the green averaged attractor. (c) Dynamics in the full $(v,w,x)$ phase space.}
   \label{fig:mlt_averaged}
\end{figure}

\subsubsection{Step 2: Normal correction}
By projecting the homological equation onto the normal bundle and solving the resulting first-order, linear, inhomogeneous ODE by variation of parameters, we obtain the following formula for the normal component of the first-order correction to the manifold
\begin{equation} \label{eq:mlt_gamma1N}
\Gamma_1^N(x,\phi) = \Phi_N(x,\phi,0) \Gamma_1^N(x,0) + \frac{1}{\omega(x)} \int_{0}^{\phi} \Phi_N(x,\phi,\sigma) \Pi_N(x,\sigma) F_1(\Gamma_0(x,\sigma))\,d\sigma,
\end{equation}
where $\Phi_N(x,\phi,\sigma) = \Pi_N(x,\tfrac{\phi}{\omega(x)}) \Phi(x,\tfrac{\phi}{\omega(x)},\tfrac{\sigma}{\omega(x)}) \Pi_N(x,\sigma)$ is the projection of the transition matrix onto the normal bundle, and the initial condition is determined by enforcing $2\pi$-periodicity in $\phi$. That is, 
\begin{equation} \label{eq:mlt_gamma1N0}
\Gamma_1^N(x,0) = \frac{1}{\omega(x)} \left( \mathbb I_3 - M_N(x) \right)^{-1} \left( \int_0^{2\pi} \Phi_N(x,2\pi,\sigma) \Pi_N(x,\sigma) F_1(\Gamma_0(x,\sigma))\, d\sigma \right),
\end{equation}
where $M_N(x) = \Pi_N(x,\tfrac{2\pi}{\omega(x)}) \Phi(x,\tfrac{2\pi}{\omega(x)},0) \Pi_N(x,0)$ is the projection of the monodromy onto the normal bundle. 

To implement the variation of parameters formula computationally, we first note that our numerically computed transition matrix from the system \eqref{eq:mlt_bvp} subject to \eqref{eq:mlt_bcs} corresponds to the matrix $\Phi(x,t,0)$. Using this, we computed the quantities in \eqref{eq:mlt_gamma1N} and \eqref{eq:mlt_gamma1N0} as follows. 
\begin{itemize}
\item For the projection, $\Phi_N(x,\phi,\sigma)$, of the transition matrix onto the normal bundle, we used the semigroup property $\Pi_N(x,\phi) = \Phi(x,\phi,\sigma) \Pi_N(x,\sigma) \Phi(x,\phi,\sigma)^{-1}$ together with the idempotency $\Pi_N^2 = \Pi_N$ to write
\[ \Phi_N(x,\phi,\sigma) 
= \Phi(x,\phi,\sigma) \,\Pi_N(x,\sigma) 
= \Phi(x,\phi,0) \, \Phi(x,\sigma,0)^{-1} \,\Pi_N(x,\sigma), 
\]
so that the integrands in \eqref{eq:mlt_gamma1N} and \eqref{eq:mlt_gamma1N0} can be expressed in terms of our already computed transition matrix. 
\item We calculated the cumulative integral in \eqref{eq:mlt_gamma1N} using the formulation of $\Phi_N(x,\phi,\sigma)$ above and MATLAB's in-built cumulative trapezoidal integration routine. 
\item The last $\phi$-data point in the previous step corresponds to the integral in \eqref{eq:mlt_gamma1N0}.
\end{itemize}
The output from evaluating \eqref{eq:mlt_gamma1N} is shown in Fig.~\ref{fig:mlt_gamma1N}.

\begin{figure}[ht!]
    \centering
    \includegraphics[width=5in]{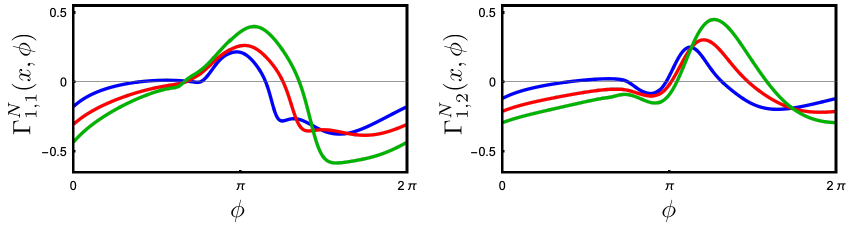}
    \put(-360,88){(a)}
    \put(-178,88){(b)}
    \caption{Components of the normal correction $\Gamma_1^N(x,\phi)$ for $x=0.09$ (blue curves), $x=0.11$ (red curves), and $x=0.13$ (green curves). (a) First component of $\Gamma_1^N(x,\phi)$ and (b) second component of $\Gamma_1^N(x,\phi)$. Due to the standard slow/fast structure of the system, i.e., since the third row of $\Phi(x,\phi,0)$ is the canonical vector $\begin{bmatrix} 0 & 0 & 1 \end{bmatrix}^T$, the third component of $\Gamma_1^N(x,\phi)$ is trivial.}
    \label{fig:mlt_gamma1N}
\end{figure}

\subsubsection{Step 3: Tangential correction}
We now return to the homological equation
\[ \mathcal{L}_0[\Gamma_1] = F_1(\Gamma_0(x,\phi)) - D\Gamma_0(x,\phi) r_1(x).  \]
The general solution is given by
\[ \Gamma_1(x,\phi) = \Phi(x,\phi,0) \Gamma_1(x,0) + J(x,\phi), \]
where 
\[ J(x,\phi) := \frac{1}{\omega(x)} \int_0^{\phi} \Phi(x,\phi,\sigma) \left( F_1(\Gamma_0(x,\sigma))-D\Gamma_0(x,\sigma) r_1(x) \right) \, d\sigma, \]
and the initial conditions $\Gamma_1(x,0)$ are to be determined. 

Recall that the transition matrix has the special structure
\[ \Phi(x,\phi,\sigma) = \begin{bmatrix} 
\Phi_{11}(x,\phi,\sigma) & \Phi_{12}(x,\phi,\sigma) & \Phi_{13}(x,\phi,\sigma) \\
\Phi_{21}(x,\phi,\sigma) & \Phi_{22}(x,\phi,\sigma) & \Phi_{23}(x,\phi,\sigma) \\ 
0 & 0 & 1 \end{bmatrix}, \]
because the Morris-Lecar-Terman model is a slow/fast system in standard form. 
Enforcing $2\pi$-periodicity of our solution, we obtain
\begin{equation} \label{eq:Gamma1periodicity}
    \begin{bmatrix} 
    1-M_{11}(x) & -M_{12}(x) \\ 
    \kappa(1-M_{11}(x)) & -\kappa M_{12}(x) 
    \end{bmatrix} 
    \begin{bmatrix} \Gamma_{1,1}(x,0) \\ \Gamma_{1,2}(x,0) \end{bmatrix} = \begin{bmatrix} J_1(x,2\pi) + \Gamma_{1,3}(x,0) M_{13}(x) \\ J_2(x,2\pi) + \Gamma_{1,3}(x,0) M_{23}(x) \end{bmatrix},
\end{equation} 
where $M(x) = \Phi(x,\tfrac{2\pi}{\omega(x)},0)$ is the monodromy matrix, $M_{ij}(x)$ denotes its $(i,j)$-element, and $\Gamma_{1,i}(x,0)$ and $J_i(x,2\pi)$ denote the $i^{\rm th}$-components of $\Gamma_1(x,0)$ and $J(x,2\pi)$. The coefficient matrix in \eqref{eq:Gamma1periodicity} is singular and the constant 
\[ \kappa = -\frac{M_{21}(x)}{1-M_{11}(x)} \]
quantifies the linear dependence between the rows. 
To guarantee that the periodicity equation \eqref{eq:Gamma1periodicity} has a solution, we require that 
\[ J_2(x,2\pi) + \Gamma_{1,3}(x,0) M_{23}(x) = \kappa \left( J_1(x,2\pi) + \Gamma_{1,3}(x,0) M_{13}(x) \right). \]
This relation uniquely determines the value for $\Gamma_{1,3}(x,0)$. 
To complete the system and specify unique values for $\Gamma_{1,1}(x,0)$ and $\Gamma_{1,2}(x,0)$, we append the additional constraint that $\Gamma_1(x,\phi)$ must be orthogonal to the basis function $\psi_1(x,\phi)$ of the adjoint null space. 

Thus, to obtain a unique solution for the unknown initial data $\Gamma_1(x,0)$, we solve the linear system
\[ \begin{bmatrix} 
1-M_{11}(x) & -M_{12}(x) & -M_{13}(x) \\ 
0 & 0 & M_{23}(x)-\kappa M_{13}(x) \\ 
\langle \Phi_1,\psi_1 \rangle & \langle \Phi_2,\psi_1 \rangle & \langle \Phi_3,\psi_1 \rangle\end{bmatrix} \begin{bmatrix} \Gamma_{1,1}(x,0) \\ \Gamma_{1,2}(x,0) \\ \Gamma_{1,3}(x,0) \end{bmatrix} = \begin{bmatrix} J_1(x,2\pi) \\ \kappa J_1(x,2\pi)-J_2(x,2\pi) \\ -\langle J(x,\phi), \psi_1(x,\phi) \rangle \end{bmatrix}, \]
where $\Phi_i$ denotes the $i^{\rm th}$ column of $\Phi(x,\phi,0)$.
The first equation comes from the periodicity condition \eqref{eq:Gamma1periodicity}, the second equation comes from the requirement that the system \eqref{eq:Gamma1periodicity} be consistent, and the third equation comes from enforcing orthogonality of $\Gamma_1(x,\phi)$ and $\psi_1(x,\phi)$. Representative solutions using this construction are shown in Fig.~\ref{fig:mlt_gamma1}. 

\begin{figure}[ht!]
    \centering
    \includegraphics[width=5in]{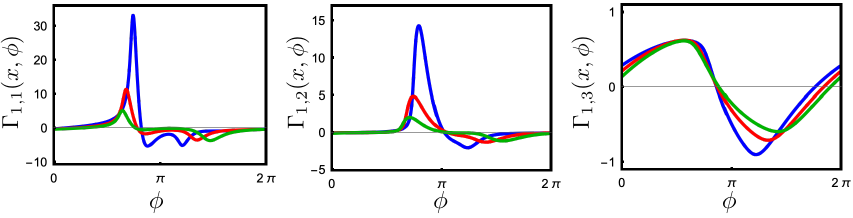}
    \put(-362,85){(a)}
    \put(-244,85){(b)}
    \put(-122,85){(c)}
    \caption{Periodic manifold corrections $\Gamma_1(x,\phi)$. The (a) first, (b) second, and (c) third components of $\Gamma_1(x,\phi)$ are shown for $x=0.09$ (blue curves), $x=0.11$ (red curves), and $x=0.13$ (green curves).}
    \label{fig:mlt_gamma1}
\end{figure}

\begin{remark}
Strictly speaking, we wish for our solution $\Gamma_1(x,\phi)$ to be orthogonal to the full set of basis functions, $\psi_i(x,\phi)$ for $i=1,2$, of the null space of the adjoint operator. 
Even though we only explicitly enforced orthogonality of $\Gamma_1(x,\phi)$ and $\psi_1(x,\phi)$, our resulting solution automatically satisfies the other orthogonality relation $\langle \Gamma_1(x,\phi), \psi_2(x,\phi) \rangle = 0$.
\end{remark}

With the manifold correction $\Gamma_1(x,\phi)$ now determined, we can extract the tangential component by projection
\[ \Gamma_1^M(x,\phi) = \Pi_M(x,\phi) \Gamma_1(x,\phi) = (\mathbb I_3 - \Pi_N(x,\phi)) \Gamma_1(x,\phi) = \Gamma_1(x,\phi)-\Gamma_1^N(x,\phi). \]
Finally, the updated manifold of periodics is given by 
\[ M_{\eps} = \Gamma_0(x,\phi) + \eps \Gamma_1(x,\phi). \] 
We show $M_{\eps}$ for a representative value of $\eps$ in Fig.~\ref{fig:mlt_meps}. 

\begin{figure}[ht!]
    \centering
    \includegraphics[width=3in]{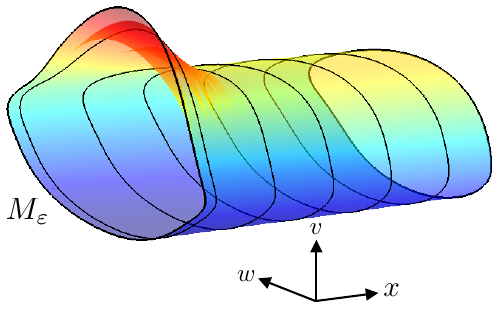}
    \caption{Manifold of (stable) periodics $M_{\eps}$ with the first two terms of the embedding for $\eps = 0.01$. The (black) curves of constant $x$ have been added for visualization purposes.}
    \label{fig:mlt_meps}
\end{figure}

Higher-order corrections, $\Gamma_j(x,\phi)$ for $j\geq 2$, can be found in a similar fashion. More specifically, the next correction $\Gamma_2(x,\phi)$ is governed by the homological equation
\[ \mathcal{L}_0[\Gamma_2] + D\Gamma_0 r_2 = F_2(\Gamma_0)+DF_1(\Gamma_0) \Gamma_1 - D\Gamma_1 r_1, \]
where the inhomogeneity consists of terms that have already been calculated (or can easily be obtained from the known data). 
Thus, a periodic solution which is orthogonal to the basis of the adjoint null space can be constructed by repeating the iterative solution procedure of Section~\ref{subsec:solnprocedure}.

\section*{Acknowledgements}
MW would like to acknowledge support from the Australian Research Council (ARC) through the Discovery Project grant scheme. BR is happy to acknowledge hospitality and financial support of the Sydney Mathematical Research Institute (SMRI) through their international visitor program.

\printbibliography
\end{document}